\documentclass[11pt,a4paper,reqno]{amsart}
\usepackage[english]{babel}
\usepackage[T1]{fontenc}
\usepackage{verbatim}
\usepackage{palatino}
\usepackage{amsmath}
\usepackage{mathabx}
\usepackage{amssymb}
\usepackage{amsthm}
\usepackage{amsfonts}
\usepackage{graphicx}
\usepackage{esint}
\usepackage{color}
\usepackage{mathtools}
\usepackage{overpic}

\usepackage[colorlinks = true, citecolor = black]{hyperref}
\newcommand{\R}{\mathbb{R}}

\newcommand{\calL}{\mathcal{L}}

\newcommand{\calH}{\mathcal{H}}

\newcommand{\calC}{\mathcal{C}}

\newcommand{\calQ}{\mathcal{Q}}

\newcommand{\diam}{\operatorname{diam}}

\def\ez{\epsilon}

\def\gz{{\gamma}}

\def\Barint_#1{\mathchoice
          {\mathop{\vrule width 6pt height 3 pt depth -2.5pt
                  \kern -8pt \intop}\nolimits_{#1}}%
          {\mathop{\vrule width 5pt height 3 pt depth -2.6pt
                  \kern -6pt \intop}\nolimits_{#1}}%
          {\mathop{\vrule width 5pt height 3 pt depth -2.6pt
                  \kern -6pt \intop}\nolimits_{#1}}%
          {\mathop{\vrule width 5pt height 3 pt depth -2.6pt
                  \kern -6pt \intop}\nolimits_{#1}}}

\numberwithin{equation}{section}

\theoremstyle{plain}
\newtheorem{thm}[equation]{Theorem}
\newtheorem*{"thm"}{"Theorem"}
\newtheorem{conjecture}[equation]{Conjecture}
\newtheorem{lemma}[equation]{Lemma}

\newtheorem{ex}[equation]{Example}
\newtheorem{cor}[equation]{Corollary}

\theoremstyle{definition}

\newtheorem{definition}[equation]{Definition}

\theoremstyle{remark}
\newtheorem{remark}[equation]{Remark}

\newcommand{\nref}[1]{(\hyperref[#1]{#1})}

\DeclareMathSymbol{\intop}  {\mathop}{mathx}{"B3}
\author{Jiayin Liu}
\title[Restricted Projections to Hyperplanes in $\mathbb{R}^n$]{Restricted Projections to Hyperplanes in $\R^n$}
\address{Mathematics Area, SISSA, Trieste, Italy}
\email{jliu@sissa.it}
\thanks{}

\begin{document}

\begin{abstract}
We study dimensions of sets projected to an $(n-2)$-dimensional family of hyperplanes in $\mathbb{R}^n$ under curvature conditions.
Let $n\ge 4$, $\Sigma \subset S^{n-1}$ be an $(n-2)$-dimensional $C^2$ manifold such that $\Sigma$ has sectional curvature $>1$ and $Z \subset \mathbb{R}^{n}$ be analytic.
If $\dim\pi_{T_yS^{n-1}}(Z) \le n-2$ for some $y \in S^{n-1}$, then
$\dim \pi_{T_xS^{n-1}}(Z) = \dim Z$ for $\mathcal{H}^{n-2}$-a.e. $x \in \Sigma$.
\end{abstract}

\maketitle
\tableofcontents

\section{Introduction}
This paper focuses on studying projection to a restricted family of
hyperplanes for any analytic set in $\R^n$. We first fix several notations.
Denote by $\pi_W: \R^n \to W$ the orthogonal projection to the ($k$-dimensional, $k=1, \cdots, n-1$) subspace $W \subset \R^n$ and by $T_x M$ the tangent space at $x \in M$ of a differentiable manifold $M \subset\R^n $. The notation $\dim$ will always refer to Hausdorff dimension.

In \cite{FasslerOrponen14}, F\"{a}ssler and Orponen conjectured that
\begin{conjecture} \label{conjfo}
   Let $Z$ be an analytic set in $\R^3$ and $\gz \in C^2([0,1], S^2)$ be a non-degenerate curve, that is
   \begin{equation}\label{nondeg0}
    {\rm span}\{ \gamma(\theta), \dot \gamma(\theta), \ddot \gamma(\theta)\}=\mathbb{R}^3 \text{  for all }  \theta \in [0,1].
\end{equation}
    Then
    \begin{enumerate}
      \item[(i)]
 \begin{equation*}
   \dim [\pi_{{\rm span}\{\gamma(\theta)\}}(Z)]  = \min\{\dim Z, 1 \}, \quad \calH^1\mbox{-\rm a.e.} \ \theta\in [0,1];
\end{equation*}
      \item[(ii)]
      \begin{equation*}
   \dim [\pi_{T_{\gamma(\theta)}S^2}(Z)]  = \min\{\dim Z, 2 \}, \quad \calH^1\mbox{-\rm a.e.} \ \theta\in [0,1].
\end{equation*}
    \end{enumerate}
\end{conjecture}
Conjecture \ref{conjfo} (i) was confirmed by Pramanik-Yang-Zahl \cite{2022arXiv220702259P} and Conjecture \ref{conjfo} (ii) was confirmed by Gan-Guo-Guth-Harris-Maldague-Wang \cite{gan2024}.

A higher
dimensional extension of Conjecture \ref{conjfo} (i) was studied in \cite{LIU2025}, which states that
\begin{thm}\label{2025}
  For $n \ge 4$, let $Z$ be an analytic set in $\R^n$ and $\Sigma \subset S^{n-1}$ be an $(n-2)$-dimensional $C^2$ manifold with sectional curvature $>1$.
    Then
 \begin{equation*}
   \dim [\pi_{{\rm span}\{x\}}(Z)]  = \min\{\dim Z, 1 \}, \quad \calH^{n-2} \mbox{-\rm a.e.} \ x\in \Sigma.
\end{equation*}
\end{thm}

This paper focuses on higher dimensional extension of Conjecture \ref{conjfo} (ii) in the spirit of Theorem \ref{2025}.
\begin{conjecture} \label{conjl}
   For $n \ge 4$, let $Z$ be an analytic set in $\R^n$ and $\Sigma \subset S^{n-1}$ be an $(n-2)$-dimensional $C^2$ manifold with sectional curvature $>1$.
    Then
      \begin{equation*}
   \dim [\pi_{T_{x}S^{n-1}}(Z)]  = \min\{\dim Z, n-1 \}, \quad \calH^{n-2} \ a.e. \ x\in \Sigma.
\end{equation*}
\end{conjecture}

\begin{remark} \label{Mattila}
  Note that if $Z \subset \R^n$ is analytic, Mattila \cite{MR0409774} showed that for all $0<s \le \min\{\dim Z, n-1\}$.
  $$  \dim \{x \in S^{n-1} : \dim[\pi_{T_{x}S^{n-1}}(Z)] < s \} \le s.$$
  Hence for every subset $\Sigma \subset S^{n-1}$ and $0<s \le \min\{\dim Z, n-1\}$,
  $$  \dim \{x \in \Sigma : \dim[\pi_{T_{x}S^{n-1}}(Z)] < s \} \le s.$$
  In particular, if $\dim Z \le n-2$ and $\Sigma\subset S^{n-1}$ is $\calH^{n-2}$-measurable, we know
  $$ \dim[\pi_{T_{x}S^{n-1}}(Z)] = \dim Z, \mbox{ for $\calH^{n-2}$ a.e. $x\in \Sigma$}. $$
  Thus Conjecture \ref{conjl} is true for the case $\dim Z \le n-2$ even without the curvature assumption.

  However, for the case $\dim Z > n-2$, one could not expect the same is true without any additional assumption on $\Sigma$. We give an example. Let $\Sigma:= S^{n-1} \cap e_n^\perp$, i.e. $\Sigma$ is the $(n-2)$-dimensional equatorial sphere lying in span$\{e_1,\cdots,e_{n-1}\}$.
  Consider $Z=$ span$\{e_1,\cdots,e_{n-1}\}$. Then $\dim Z =n-1$. We have
  $$\pi_{T_xS^{n-1}}(Z)= Z \cap T_xS^{n-1} \mbox{ for all }x \in \Sigma.$$
  Since $Z^\perp =$ span$\{e_n\} \subset T_xS^{n-1}$
  for all $x \in \Sigma$, we deduce that $\dim[\pi_{T_xS^{n-1}}(Z)]=n-2$ for all $x \in \Sigma$. Thus Conjecture \ref{conjl} does not hold for this choice of $\Sigma$, which has sectional curvature $\equiv 1$.
\end{remark}

The main result of this paper is partial progress towards Conjecture \ref{conjl} when $\dim Z >n-2$. Note that the non-degenerate condition \eqref{nondeg0} is equivalent to $\gamma$ having non-vanishing geodesic curvature (c.f. \cite[Page 3]{LIU2025}).

%

\begin{thm}\label{cormain}
For $n \ge 3$, let $Z$ be an analytic set in $\R^n$ and $\Sigma \subset S^{n-1}$ be an $(n-2)$-dimensional $C^2$ manifold having
\begin{itemize}
  \item non-vanishing geodesic curvature if $n=3$;
  \item sectional curvature $>1$ if $n \ge 4$.
\end{itemize}
Moreover, if there exists $y \in S^{n-1}$ such that
   \begin{equation}\label{le1}
     \dim[\pi_{T_{y}S^{n-1}}(Z)] \le n-2,
   \end{equation}
  then,
  \begin{equation*}
   \dim [\pi_{T_{x}S^{n-1}}(Z)]  = \min\{\dim Z, n-1 \}, \quad \calH^{n-2} \mbox{-\rm a.e.} \ x \in \Sigma.
\end{equation*}
\end{thm}

\begin{remark} \label{rmkmain}
  \rm
(i) Notice that for sets $Z$ with $n-2<\dim Z \le n-1$, it is easy to see that $\dim[\pi_{W}(Z)] \ge \dim Z-1$ for all hyperplanes $W$. Since $\dim Z-1 \le n-2$, assumption \eqref{le1} is realisable.

(ii) When $n=3$, as mentioned above, Gan-Guo-Guth-Harris-Maldague-Wang \cite[Theorem 1]{gan2024} proved Theorem \ref{cormain} for any analytic $Z\subset \R^3$ without requiring assumption \eqref{le1} using Fourier analysis. Earlier work \cite[Theorem 3.2]{Jarvenpaa2008} (see also \cite[Proposition 1.5]{FasslerOrponen14}) proved Theorem \ref{cormain} for the case $\dim Z \le 1$ in $\R^3$ using the sublevel set technique pioneered by Kaufman \cite{Kaufman1968}. This paper gives a new Fourier analysis free proof under assumption \eqref{le1} using incidence geometry. It would be interesting to further prove Theorem \ref{cormain} without assuming \eqref{le1} in both Fourier analysis method and incidence geometry method.

\end{remark}

\begin{ex} \label{eg1}
   \rm
   (i) Circles in $S^2$ that are not equators have non-vanishing geodesic curvature. Similarly, $(n-2)$-dimensional spheres that are not equatorial spheres have sectional curvature $>1$.
   In particular, let $\Sigma_c: = S^{n-1} \cap \{x_{n}=c\}$ for $c \in (0,1)\cup(-1,0)$. Indeed, it is straightforward to check that $\Sigma_c$ has geodesic curvature $\pm c/\sqrt{1-c^2} \ne 0$ (depending on the orientation of $\Sigma_c$) when $n=3$ and
   sectional curvature $1/(1-c^2) >1$ when $n \ge 4$.

   (ii) Other examples include small $C^2$-perturbations of circles and spheres in (i), since the curvature condition in Theorem \ref{cormain} is stable under small perturbations.
\end{ex}

\subsection{Background of (Restricted) Projections}
Projection properties were first established by Marstrand \cite{MR0063439} who showed that the orthogonal projection of any Borel set $Z \subset \R^2$ to almost every line has dimension $\min\{1, \dim Z\}$. For all $n \ge 2$ and $1 \le m \le n$, Mattila \cite{MR0409774} showed that the orthogonal projection of any Borel set $Z \subset \R^n$ to almost every $W \in G(n,m)$ has dimension $\min\{m, \dim Z\}$ where $G(n,m)$ is the $m(n-m)$-dimensional Grassmann manifold of all
$m$-dimensional subspaces of $\R^n$ and ``almost every'' refers to the $m(n-m)$-dimensional Hausdorff measure on $G(n,m)$.

Restricted projection focuses on studying orthogonal projection of a set in $\R^n$ to a family of $m$-subspaces which forms a submanifold $\Sigma \subset G(n,m)$. Noting that when $m=1$ or $m=n-1$, $G(n,m)$ can be naturally identified with the unit sphere $S^{n-1} \subset \R^n$. In particular, when $n=3$, $\dim G(3,1) = \dim G(3,2) = 2$, thus the only meaningful submanifold $\Sigma \subset G(3,m)$ ($m=1,2$) to consider in the restricted projection is $1$-dimensional, which is a curve in $S^2$. These cases are relatively well understood by the aforementioned Conjecture \ref{conjfo} and its resolution in \cite{2022arXiv220702259P,gan2024}, together with other previous partial progress \cite{Jarvenpaa2008,FasslerOrponen14,Jarvenpaa2014,Chen2018,MR4148151,
OrponenVenieri2020,Harris2022,Harris2023,marstrandtype}, etc.
For the case $\Sigma \subset G(n,m)$ being a $C^\infty$-curve, certain restricted projection property was established by Gan-Guo-Wang \cite{gan2022restricted}.
In general, when $2 \le m \le n-2$ and $2 \le \dim \Sigma \le \dim G(n,m)$, the study of restricted projection remains widely open.
Very recently, progress was made for the case $G(4,2)$ in \cite{cohen2026spreadoutfurstenbergsetsii}.


\subsection{Ideas of the Proof of Theorem \ref{cormain}}
Let $Z \subset B(o,1/2) \subset \R^n$ be an analytic set with
$$  t:=\dim Z \in (n-2,n-1].   $$
By proving the theorem on each local chart of $\Sigma$, we can assume $\Sigma \subset \mathbb{R}^{n}$ is a $C^2$ diffeomorphism from $[0,1]^{n-2}$ to $\Sigma([0,1]^{n-2})$ and therefore throughout the paper we assume $\Sigma \in C^2([0,1]^{n-2}, \mathbb{R}^{n})$ to be a $C^2$ diffeomorphism.
For each $x \in [0,1]^{n-2}$, we choose an orthonormal basis $\{E_1(x), \cdots E_{n-1}(x)\}$ of $T_{\Sigma(x)}S^{n-1}$ such that $E_i$ is $C^2$ continuous with respect to $x$. 
Thus, the projection $\pi_{T_{\Sigma(x)}S^{n-1}}$ can be written as
$$  \pi_{T_{\Sigma(x)}S^{n-1}}(z)=(E_1\cdot z, \cdots,   E_{n-1}\cdot z), \quad z \in \R^n.$$

If $z \in B(o,1/2)$ is fixed, then the map
\begin{equation}\label{fw}
  [0,1]^{n-2} \ni x \mapsto f_z(x):=(E_1(x)\cdot z, \cdots,  E_{n-1}(x)\cdot z) \in [0,1]^{n-1}
\end{equation}
maps $[0,1]^{n-2}$ into $[0,1]^{n-1}$ (up to a translation).
The graph $f_z$ of this map is an $(n-2)$-dimensional manifold in $[0,1]^{2n-3}$:
$$ \mbox{graph}(f_z)(x):=(x_1, \cdots, x_{n-2},E_1(x)\cdot z, \cdots, E_{n-1}(x)\cdot z) \quad x \in [0,1]^{n-2}.$$
For instance, if $n=3$, $\mbox{graph}(f_z)$ is a curve in $[0,1]^3$ for each $z \in Z$.
Writing $P_x=T_{\Sigma(x)}S^{n-1}$ and observing that
$$ \bigcup_{x \in [0,1]^{n-2}}\{x\}\times\pi_{P_x}(Z) = \bigcup_{x \in [0,1]^{n-2}}\bigcup_{z \in Z }\{x\}\times\pi_{P_x}(z) = \bigcup_{z \in Z }\bigcup_{x \in [0,1]^{n-2}}(x,\pi_{P_x}(z))=\bigcup_{z \in Z}\mbox{graph}(f_z), $$
we see that the study of restricted projection can be transferred to the study of the union of graphs of $\{f_z\}_{z \in Z}$.

To be precise,
the fact that Theorem \ref{cormain} holds is equivalent to that
\begin{equation}\label{heu11}
  \dim (\pi_{P_x}(Z)) = \dim (\bigcup_{z \in Z} [(\{x\} \times \R^{n-1}) \cap \mbox{graph}(f_z)])  =t, \ \mathcal{H}^{n-2}\mbox{-a.e. } x \in [0,1]^{n-2}.
\end{equation}
For each $z \in Z$, let $X_z \subset [0,1]^{n-2}$ with $\calL^{n-2}(X_z) \sim 1$ and consider the set
\begin{equation}\label{heu31}
    \mathcal{Z}: = \bigcup_{z \in Z} [(X_z \times [0,1]^{n-1}) \cap \mbox{graph}(f_z)] = \bigcup_{z \in Z} \bigcup_{x \in X_z}(x,f_z(x)) \subset [0,1]^{2n-3}.
\end{equation}
Then \eqref{heu11} will be verified by showing that $\mathcal{Z}$ has dimension
\begin{equation}\label{heu21}
  \dim \mathcal{Z} = n-2 + \dim Z = n-2 + t,
\end{equation}
which will be further guaranteed by the following lemma.

Given $\eta>0$ and a subdomain $D \subset [0,1]^{m}$, the vertical $\eta$-neighborhood of the graph of $f:D \to \R^l$ is
\begin{equation}\label{vertical}
  f^{\eta}(D):= \{(x,y) \in D \times \R^{l} : |(x,y)-(x,f(x))|= |y-f(x)|< \eta \} \subset \R^{m+l} .
\end{equation}

\begin{lemma}\label{volume2}
For any $\epsilon>0$, there exists $\delta_{0}=\delta_{0}(\epsilon) \in (0,\tfrac{1}{2}]$ such that the following holds for all $\delta \in (0,\delta_{0}]$:
$$|\mathcal{Z}|_{\delta} \gtrsim \delta^{9\epsilon - (n-2 +t)}.$$
\end{lemma}
We give the heuristic idea to show Lemma \ref{volume2}. We choose a $\delta$-separated subset of $Z$ (still denoted by $Z$) with cardinality $\sim \delta^{-t}$ and in the following of this section $Z$ will always refer to this discrete set .
It suffices to show
\begin{equation}\label{thisbd}
   \calL^{2n-3}(\bigcup_{z \in Z}f_z^\delta([0,1]^{n-2})) \gtrsim    \delta^{O(\epsilon) + 2n-3-(t+ n-2)} = \delta^{O(\epsilon) + n-1-t}.
\end{equation}
To this end, we first need to bound from above the volume of the intersection $f^\delta_z([0,1]^{n-2}) \cap f^\delta_{z'}([0,1]^{n-2})$ of two $\delta$-vertical neighbourhoods.
We have the following estimate.

\begin{lemma} \label{pairvolume2}
We have
  $$ \calL^{2n-3} (f^\delta_z([0,1]^{n-2}) \cap f^\delta_{z'}([0,1]^{n-2})) \lesssim \frac{\delta^{2n-3}}{\|f_z-f_{z'}\|_{C^2([0,1]^{n-2})}^{n-2}} $$
  where $\calL^{d}$ is the $d$-dimensional Lebesgue measure.
\end{lemma}
In fact, we will show the estimate holds for a more general family of maps and also show Lemma \ref{pairvolume2} holds in a more general setting. See Lemma \ref{fg} in Section \ref{s2}.

The estimate in Lemma \ref{pairvolume2} shows that the intersection looks like the $\delta$-neighbourhood of the intersection of two $(n-2)$-dimensional planes intersecting at a point (e.g. span$\{x_1,x_2\}$  and span$\{x_3,x_4\}$ in $\R^5$ when $n=4$).

However, only the volume estimate in Lemma \ref{pairvolume2} is not enough for deducing the lower bound in \eqref{thisbd}. We also need to control the total number of pairs $(z,z') \in Z\times Z$ such that $f_z^\delta([0,1]^{n-2}) \cap f_{z'}^\delta([0,1]^{n-2}) \ne \emptyset$. This is because when $t > n-2$, the total number of pairs in $Z\times Z$ is $\delta^{-2t} > \delta^{-2(n-2)}$ which is large.
Hence,
we will additionally establish a line-cone incidence estimate together with the help of the assumption \eqref{le1} in Theorem \ref{cormain}. We sketch the ideas below.

Note that $f_z = f_{z'}$ at $x$, i.e. $f_z(x)=f_{z'}(x)$ is equivalent to that
$$  \pi_{P_x}(z) = \pi_{P_x}(z'),$$
which, together with $P_x^\perp = (T_{\Sigma(x)}S^{n-1})^\perp= $span$\{\Sigma(x)\}$ implies that
$$  z' = z + u_z \Sigma(x) \quad \mbox{ for some $u_z \ne 0$},$$
since all points in $\R^n$ which have same image as $z$ under $\pi_{P_x}$ are exactly the line $z + \R \Sigma(x)$.
This shows that for a fixed point $z \in Z$, all other points $z' \in Z \subset B(o,1/2)$ such that $f_z \cap f_{z'} \ne \emptyset$ are contained in
$$  \calC_z:=  \bigcup_{x \in [0,1]^{n-2}, r \in [-1,1] } z + r \Sigma(x)= : \bigcup_{x \in [0,1]^{n-2}} L_{z,x},  $$
which is a cone with apex $z$ and cross-section $\Sigma$.
If $\Sigma$ is the $(n-2)$-sphere $\Sigma_c$ in Example \ref{eg1} (i), $\calC_z$ is a circular cone in $\R^n$ with opening angle $\arccos c$.
We call the line segments $\{L_{z,x}\}_{x\in [0,1]^{n-2}}$ generatrices of $\calC_z$. Obviously, $\calC_z$ is the union of all its generatrices.
We conclude that
\begin{equation*}
  f_z \cap f_{z'} \ne \emptyset \quad \Longleftrightarrow \quad z' \in \calC_z.
\end{equation*}
Thus in the $\delta$-discretised version, we will prove the following relation: for any $z \in Z$,
\begin{equation}\label{linecone0}
  |\{z' \in Z : f_z^\delta([0,1]^{n-2}) \cap f_{z'}^\delta([0,1]^{n-2}) \ne \emptyset \}| \sim |\{z' \in Z : z' \in \calC_z^\delta \}|.
\end{equation}
where $A^\delta$ denotes the $\delta$-neighbourhood in $\R^n$ for any subset $A \subset \R^n$.
We have the following theorem on the incidence of a line and the cone $\calC_x$. Recall that the angle of a line and a subspace in $\R^n$ is the angle of the line and its orthogonal projection to that subspace.
\begin{thm} \label{lineconen}
Let $\Sigma$ be as in Theorem \ref{cormain}.
If a line $L \subset \R^n$ satisfies $\min_{p \in \calC_z}\{\angle(L,T_p\calC_z)\}  \ge a$ for some $a \gg \delta >0$, then
  $$\calL^{n}(\calC_z^\delta \cap L^\delta) \lesssim_a \delta^n. $$
  where $\angle(L,T_p\calC_z)$ is the Euclidean angle between $L$ and the tangent space $T_p\calC_z$.
\end{thm}
The proof of this theorem is geometric and relies on the curvature condition of $\Sigma$ in Theorem \ref{cormain}.
As a consequence of Theorem \ref{lineconen}, recalling that $Z$ is $\delta$-separated, we know that
\begin{equation}\label{less1}
  |\calC_z^\delta \cap L^\delta \cap Z|\lesssim_a 1, \quad \forall z \in Z.
\end{equation}

Finally, under the assumption \eqref{le1}, that is,
\begin{equation}\label{le11}
  s : = \dim[\pi_{T_yS^{n-1}}(Z)] \le n-2 < t,
\end{equation}
we can efficiently upper bound the quantity in \eqref{linecone0} as follows.

Let $Z_y:=\pi_{T_yS^{n-1}}(Z)$. By \eqref{le11}, we find a maximal $\delta$-separated $(\delta,s)$-set $Z'_y \subset Z_y$ with $|Z'_y| \sim \delta^{-s}$ (c.f. Definition \ref{dst} in Section \ref{s5} for $(\delta,s)$-set). For each $w \in Z'_y$, the preimage set
$$W_w:=\pi_{T_yS^{n-1}}^{-1}(\{w\}) \cap Z$$
  is contained in the line $L_{w,y} = w+ \R y$.
Since $\{w+ \R y\}_{w \in Z'_y}$ are parallel, we deduce that
\begin{equation}
  Z  \subset \bigcup_{w \in Z'_y} L_{w,y}^\delta.
\end{equation}
Roughly speaking, we are able to use condition \eqref{le11} to partition $Z$ into subsets contained $\{W_w^\delta\}_{w \in Z'_y}$ which further contained in an $s$-family of parallel $\delta$-tubes $\{L_{w,y}^\delta\}_{w \in Z'_y}$. This, combining with \eqref{less1} implies that
$$ |\calC_z^\delta \cap Z| \le \sum_{w \in Z'_y} |\calC_z^\delta \cap L_{w,y}^\delta \cap Z| \lesssim_a |Z'_y| \lesssim \delta^{-s} \le \delta^{-(n-2)}. $$
Here in this introduction we omit the details on a technical approximation process to guarantee $\min_{p \in \calC_z}\{\angle(L,T_p\calC_z)\}  \ge a$.
Recalling \eqref{linecone0}, this estimate is sufficient to deduce Lemma \ref{volume2} and then Theorem \ref{cormain}, which completes the outline of the proof of Theorem \ref{cormain}.


\bigskip
\noindent
{\bf Organization.}
Section \ref{s2} is mainly designed to show Lemma \ref{pairvolume2}.
Section \ref{s4} is devoted to showing Theorem \ref{lineconen}.
Finally, we complete the paper by proving Theorem \ref{cormain} in Section \ref{s5}.

\bigskip
\noindent
{\bf Notation.}
The notation $A \lesssim B$ means that there exists a constant $C \geq 1$ such that $A \leq CB$. The two-sided inequality $A \lesssim B \lesssim A$ is abbreviated to $A \sim B$. The constant $C$ can depend on the dimension $n$, and if we fix a cinematic family $F$ of cinematic constant $K$ and doubling constant $D$, $C$ may depend on $K$ and $D$.
If we want to stress the dependence of the constant $C$ with some  parameter "$\theta$", we indicate this by writing $A \lesssim_{\theta} B$. The constant $C$ may change from line to line in the proof. For simplicity, we write $\log=\log_2$ throughout the paper.

By the notation $|A|$ we denote the cardinality of $A$ if $A$ is a set and we denote the Euclidean norm if $A$ is a point. By $\mathcal{H}^{s}(A)$ (resp. $\mathcal{L}^{k}(A)$) the $s$-dimensional Hausdorff measure (resp. $k$-dimensional Lebesgue measure) of $A$. For $\delta>0$, $A^\delta$ stands for the $\delta$-neighbourhood and $|A|_\delta$ stands for the $\delta$-covering number of $A$. For a set $A \subset \R^n$ and a point $p \in \R^n$, $A \pm \{p\}$ is the set $\{x \pm p: x \in A \}$.



\bigskip
\noindent
{\bf Acknowledgement.}
The author would like to thank Tuomas Orponen for his interest in this paper and providing some references.

\bigskip
\noindent
{\bf Disclosure of AI use.} The first version of this paper appeared in April 2026 without AI use. For the current version, the author used AI tools to realise that Theorem 1.5 in the first version can be seen as a corollary of Mattila's projection theorem as explained in Remark \ref{Mattila}.
Also, the author used AI tools to
simplify several proofs and write technical Lemma \ref{int2} and Lemma \ref{frost}. All the proofs are checked and rewritten by the author and the author takes full responsibility of the paper.

\section{Cinematic Maps and Their Intersections} \label{s2}
We define a class of cinematic maps and show that the intersection of any two cinematic maps
satisfies the upper bound in Lemma \ref{pairvolume2}. Then, we check that the maps $\{f_z\}_{z \in Z}$ arising from the restricted projection are cinematic maps.

\subsection{Cinematic Maps}
\begin{definition}[Cinematic Map]\label{cine}
  Let $F \subset C^2(U,R^l)$ where $U \subset \mathbb{R}^k$ is a domain, $k,l \in \mathbb{N}$ with $l \ge k$. We say $F$ is a cinematic family of
maps, with cinematic constant $K$ and doubling constant $D$ if the following conditions
hold:
\begin{enumerate}
  \item $F$ is contained in a ball of diameter $K$ under the usual metric on $C^2(U,R^l)$;
  \item $F$ is a doubling metric space, with doubling constant at most $D$;
  \item For all $f, g \in F$, we have
  $$  \inf_{(x,\xi) \in U \times S^{k-1}} \{ |f(x) - g(x)| + |\nabla_\xi f(x) -\nabla_\xi g(x)| \} \ge K^{-1} \|f-g\|_{C^2(U)}  $$
  where $\nabla_\xi h= \nabla h \cdot \xi$ for $h=f,g$, noting that $\nabla h$ is a linear map from $\R^{k}$ to $\R^l$ and $\xi \in S^{k-1} \subset \R^{k}$ is considered as a direction of $\R^{k}$.
\end{enumerate}
\end{definition}
\begin{remark} \rm
We make a remark about (3) in Definition \ref{cine}. Let us assume $f \ne g$ and $|f(x) - g(x)|=0$ at some $x$. Then (3) implies that
$$\inf_{\xi \in  S^{k-1}}|\nabla_\xi f(x) -\nabla_\xi g(x)| \sim_K \|f-g\|_{C^1(U)}.$$
This condition guarantees that whenever $f(x)=g(x)$, the smallest singular value of $\nabla(f-g)(x)$ is not zero, i.e. $\nabla(f-g)(x)$ as an $(l \times k)$-matrix has full rank.
\end{remark}

\begin{lemma} \label{local}
Assume $\eta>0$, $n \ge 3$.
Let
$F \subset C^{2}([0,1]^{n-2},[0,1]^{n-1})$ be a cinematic family with cinematic constant $K$ and $f,g \in F$.
  If $U \subset [0,1]^{n-2}$ is a (closed) dyadic cube such that its diameter is smaller than $\frac\eta{2K}$, and $k$ is one of the functions $|f-g|, |\nabla_\xi (f-g)|$ for some $\xi \in S^{n-3}$, then $k$ satisfies either one of the following two inequalities:
   \begin{itemize}
     \item[(S)] $ k(x) < \eta\|f-g\|_{C^2([0,1]^{n-2}) } \mbox{ for all } x \in U$.
     \item[(L)] $k(x) \ge \eta \|f-g\|_{C^2([0,1]^{n-2})}/2 \mbox{ for all } x \in U$.
   \end{itemize}
\end{lemma}

\begin{proof}
 Write $d=\|f-g\|_{C^2( [0,1]^{n-2}) }$.
 Since $F$ is contained in a ball of diameter $K$, we know
 $d \le K$. In particular, we know Lip$(f-g)(x) \le K$ and Lip$|\nabla(f-g)|(x)\le K$ for all $x \in U$. We have
 $$  |k(x)-k(x')| \le \max_{y\in U}\mbox{{\rm Lip} } k(y)\cdot|x-x'| \le K\diam U < \frac{\eta}2, \quad \forall x,x' \in U,  $$
 where $k$ is one of the  functions $|f-g|, |\nabla_\xi (f-g)|$.

 Next,  suppose that the
alternative (L) fails for $k = f - g$. So there is $x_0 \in U$ such that $|k(x_0)| < \eta d/2$. And for any $x \in U$, since $|x-x_0| < \eta/(2K)$, by defining $l= k/d$, we deduce that Lip$_{[0,1]^{n-2}}l \le 1$ and
$$ \frac{|k(x)|}{d} \le \frac{|k(x_0)|}{d}+ \left| \frac{|k(x)|}{d} - \frac{|k(x_0)|}{d}\right| < \frac{\eta}{2} + |l(x)-l(x_0)| \le \frac{\eta}{2} +\mbox{ [Lip$_{[0,1]^{n-2}}l$}] |x-x_0|  \le \eta.$$
 Thus the alternative (S) holds for $k$. The proof of the case $k = |\nabla_\xi (f-g)|$ is the
same.
\end{proof}

The above lemma gives the following corollary.

\begin{cor} \label{uniq}
  Let
$F \subset C^{2}([0,1]^{n-2},[0,1]^{n-1})$ be a cinematic family of maps with cinematic constant $K$, $f,g \in F$ and $h:=f-g$. Then for each dyadic cube $U \subset [0,1]^{n-2}$ with ${\rm diam \, } U < \frac{1}{4K^2}$,

{\rm (A):} at least one of the following two conditions hold:
\begin{enumerate}
  \item[(i)]
$|h(x)| \ge \frac{1}{4K} \|h\|_{C^2([0,1]^{n-2}) }$ for all $x \in  U$.
  \item[(ii)]
$|\nabla_\xi h(x)| \ge \frac{1}{2K} \|h\|_{C^2( [0,1]^{n-2}) }$ for all $x \in  U$ and $\xi \in S^{n-3}$.
\end{enumerate}

{\rm (B):} Moreover, if (A)(ii) holds, then
\begin{equation}\label{hxy}
  |h(y)-h(x)| \ge \frac{\|h\|_{C^2( [0,1]^{n-2}) }}{4K}|x-y|, \quad \forall x,y \in U.
\end{equation}
\end{cor}

\begin{proof}
We first show (A):
   We apply Lemma \ref{local} to $h$ with $\eta = 1/(2K)$. As a result, we know for each dyadic cube $U \subset [0,1]^{n-2}$ with diameter ${\rm diam \, } U < \frac{1}{4K^2} $,
 either (S) or (L) holds for $k$ and all $x \in U$, where $k$ is one of the
functions $h$, or $\nabla_\xi h$. On the other hand,
since $F$ is a cinematic family of functions, if the alternative (L) does not hold for the functions $h$, then by Definition \ref{cine} (3), for all $\xi \in S^{n-3}$, (L) holds for $\nabla_\xi h=\nabla_\xi (f-g)$.

We show (B). Fix $x,y \in U$. Note that if $x=y$, \eqref{hxy} holds trivially. We thus assume
$x \ne y$ and write
 $$ y = x+ \xi_y |y-x|, \mbox{ where } \xi_y = \frac{y-x}{|y-x|} \in S^{n-3}$$
 By Definition \ref{cine} (1), we know that
$$d:=\|h\|_{C^2([0,1]^{n-2})} \le 2K < \infty$$
 and hence for any $s \in [0,|x-y|]$, we have
$$  |\nabla_{\xi_y}h(x+s\xi_y) - \nabla_{\xi_y}h(x)| \le {\rm Lip}(\nabla_{\xi_y} h)\cdot |x+s\xi_y-x| \le  ds.   $$
Let
$$ v(s):=\nabla_{\xi_y}h(x+s\xi_y) - \nabla_{\xi_y}h(x), s \in [0,|x-y|]. $$
We deduce that
\begin{equation}\label{vs}
 |v(s)| \le ds \le d|x-y|, \quad s \in [0,|x-y|].
\end{equation}
We compute
\begin{align*}
  |h(y)-h(x)| = |\int_0^{|x-y|}\nabla_{\xi_y}h(x+s\xi_y) \, ds | \\
   & =|\int_0^{|x-y|}\nabla_{\xi_y}h(x) \, ds + \int_0^{|x-y|}v(s) \, ds |  \\
   &\ge |\nabla_{\xi_y}h(x)||x-y| - |\int_0^{|x-y|}v(s) \, ds | .
\end{align*}
Since $\nabla_{\xi_y}h$ satisfies (A)(ii), we have $|\nabla_{\xi_y}h(x)| \ge \frac{d}{2K}$. Combining \eqref{vs}, we deduce that
$$ |h(y)-h(x)| \ge \frac{d}{2K}|x-y| - d|x-y|^2 = d|x-y|(\frac{1}{2K} - |x-y|).  $$
Since $\diam U < \frac{1}{4K^2}$, we know $\frac{1}{2K} - |x-y| > \frac{1}{4K}$, which gives \eqref{hxy} as desired.

  The proof is complete.
\end{proof}

Recall from \eqref{vertical} that for arbitrary $\eta>0$ and
open set $D \subset [0,1]^{n-2}$, the vertical $\eta$-neighborhood of the graph of $f:D \to \R^l$ is
   $$f^{\eta}(D):= \{(x,y) \in D \times \R^{l} : |(x,y)-(x,f(x))|= |y-f(x)|< \eta \}.$$

\begin{lemma}\label{nbhd}
  If $f \in C^1(D , \R^l)$, then
   $$\calL^{n-2+l}(f^{\eta}(D)) = \calL^{n-2}(D).$$
\end{lemma}
\begin{proof}
  Note that
  $$ f^{\eta}(D) = \bigcup_{|\eta'|<\eta} (D,f(D)+\eta'). $$
 Write $f=(f_1, \cdots, f_l)$ where $f_i \in C^1(D)$, $i=1,\cdots,l$.
 We furthermore have
 $$ f^{\eta}(D) = \bigcup_{|\eta'|<\eta} (D,f_1(D)+\eta', \cdots, f_l(D)+\eta'). $$
  By Fubini's Theorem, we deduce that
  $$ \calL^{n-2+l}(f^{\eta}(D)) = \eta^l \omega_l\cdot  \calL^{n-2}(D) $$
  where $\omega_l$ is the volume of $l$-dimensional unit ball.
  This completes the proof.
\end{proof}

\begin{lemma} \label{fg}
Let $\delta>0$,
$F \subset C^{2}([0,1]^{n-2},[0,1]^{n-1})$ be a cinematic family with cinematic constant $K$. Then, for any $f,g \in F$ with $d:=\|f-g\|_{C^2([0,1]^{n-2})} \ge \delta$, we have
\begin{equation}\label{meas}
  \calL^{2n-3}[f^{\delta}([0,1]^{n-2})\cap g^{\delta}([0,1]^{n-2})] \lesssim \frac{\delta^{2n-3}}{\|f-g\|_{C^2([0,1]^{n-2})}^{n-2}}.
\end{equation}
\end{lemma}

\begin{proof}
By
decomposing $[0,1]^{n-2}$ into dyadic cubes $\{U\}$ with ${\rm diam \, } U < \frac{1}{4K^2}$ to prove \eqref{meas}, it suffices to show
$$ \calL^{2n-3}(f^\delta(U) \cap g^\delta(U)) \lesssim \frac{\delta^{2n-3}}{d^{n-2}}   $$
for each $U \subset[0,1]^{n-2}$.
In the following, we fix $U$ and write $f^\delta \cap g^\delta = f^\delta(U) \cap g^\delta(U)$ for simplicity.

  If $f^\delta \cap g^\delta = \emptyset$, then there is nothing to prove. Also, since $d \in [\delta,1]$, if $\delta \le d \le 8K\delta$,
   then using Lemma \ref{nbhd}, we know that
   $$\calL^{2n-3}(f^\delta \cap g^\delta) \le \calL^{2n-3}(f^\delta) \sim \delta^{n-1} \sim \frac{\delta^{n-1}d^{n-2}}{d^{n-2}} \sim \frac{\delta^{2n-3}}{d^{n-2}} .$$
Thus in the following, we assume
\begin{equation}\label{ass1}
  d > 8K\delta.
\end{equation}

  Assume $f^\delta \cap g^\delta \ne \emptyset$.
  Let $P_{f,g}^U$ be the orthogonal projection of
  $f^\delta \cap g^\delta$ from $U \times [0,1]^{n-1}$ to $U$.
  We have
  $$ f^\delta \cap g^\delta \subset f^\delta(P_{f,g}^U). $$
  Thus
 by Lemma \ref{nbhd},
 we need to show
  \begin{equation}\label{Pfg}
   \mathcal{L}^{n-2} (P_{f,g}^U) \lesssim \frac{\delta^{n-2}}{d^{n-2}}.
  \end{equation}

  Write $h=f-g$. By $f^\delta \cap g^\delta \ne \emptyset$, we know there exists $x \in U$ such that $|h(x)|\le 2\delta < \frac{d}{4K}$ which implies $|h|$ does not enjoy Corollary \ref{uniq} (i).
  Hence Corollary \ref{uniq} (A)(ii) and (B) hold, which gives
  $$  |h(x)-h(y)| \ge \frac{d}{4K} |x-y|.  $$
  Note that $z \in P_{f,g}^U$ implies that $|h(z)|=|f(z)-g(z)| \le 2 \delta$. Combining the above inequality, for any $x,y \in P_{f,g}^U$, we deduce that
  $$ \frac{d}{4K} |x-y| \le |h(x)-h(y)| \le |h(x)| + |h(y)| \le 4\delta,$$
  which implies that
  $$ \diam P_{f,g}^U \le \frac{16K\delta}{d}$$
  and furthermore, implies \eqref{Pfg} as desired.

%
%
    The proof of Lemma \ref{fg} is complete.
\end{proof}

\subsection{$\{f_z\}_{z \in Z}$ is a Cinematic Family}

\begin{lemma} \label{cineprop}
  Let $Z$ be a subset in $B(o,1/2)\subset \mathbb{R}^{n}$ and $\Sigma : [0,1]^{n-2} \to \mathbb{S}^{n-1}$ be as in Theorem \ref{cormain}. Then  $F:=\{f_{z} \in C^2( [0,1]^{n-2},[0,1]^{n-1})\}_{z \in Z}$ defined in \eqref{fw} is a cinematic family with the doubling constant comparable to that of  $\mathbb{R}^{n}$. In particular, $F$ is biLipschitz homeomorphic to $Z$.
\end{lemma}

\begin{proof}
First we show (1) in Definition \ref{cine}.
Since
\begin{equation}\label{cine1}
  \sup_{f_{z} \in F }\|f_{z}\|_{C^2([0,1]^{n-2})} \le C_\Sigma \max_{z \in Z}|z| \lesssim 1
\end{equation}
 where the implicit constant only depends on $\Sigma$, (1) holds.

Fix $y,z \in Z$. We show (3) in Definition \ref{cine}. Similar to \eqref{cine1}, we have
\begin{equation}\label{dou1}
  \|f_y-f_{z}\|_{C^1([0,1]^{n-2})} \le \|f_y-f_{z}\|_{C^2([0,1]^{n-2})}\lesssim |y-z|.
\end{equation}
Write
$h:= f_y-f_{z}$. Letting $p=y-z$, we prove (3) by showing
 that there exists a constant $c_\Sigma>0$ such that
\begin{equation}\label{cine2}
|h(x)|+|\nabla_\xi h(x)|
\geq c_\Sigma |p|
\end{equation}
for every $x\in[0,1]^{n-2}$ and every $\xi\in S^{n-3}$.

Since $\Sigma$ is a $C^2$-diffeomorphism,
\begin{equation}\label{m0}
  m_0:=
\min_{\substack{x\in [0,1]^{n-2}\\ \xi\in S^{n-3}}}
|\nabla_\xi\Sigma(x)|>0.
\end{equation}
For each $x \in [0,1]^{n-2}$, consider the linear map $A_x:\mathbb [0,1]^n\to\mathbb [0,1]^{n-1}$ defined by
\[
A_x(p):=
\bigl(E_1(x)\cdot p,\ldots,E_{n-1}(x)\cdot p\bigr).
\]
Since $\{E_i(x)\}_{1 \le i \le n-1}$ forms an orthonormal basis of
$T_{\Sigma(x)}S^{n-1}=\Sigma(x)^\perp$, we have
\begin{equation}\label{qq}
  A_x(\Sigma(x))=0,
\mbox{ and }
|A_xq|=|q|
\quad\text{for every }q\in T_{\Sigma(x)}S^{n-1}.
\end{equation}
Also, let
\[
B_0:=
\sup_{\substack{x\in [0,1]^{n-2}\\ \xi\in S^{n-3}}}
|\nabla_\xi A_x|<\infty.
\]
Write
$p:=y-z$, thus
\[
h(x):=f_y(x)-f_{z}(x)=A_x(p), \quad \forall x \in [0,1]^{n-2}.
\]
Decompose $p$ as
\[
p=a\Sigma(x)+q
\mbox{ where }
a=p\cdot \Sigma(x),
\
q\in (\Sigma(x))^\perp.
\]
Differentiating the identity $A_x(\Sigma(x))=0$ in the
direction $\xi$ yields
\begin{equation}\label{de2}
  (\nabla_\xi A_x)\Sigma(x)=-A_x(\nabla_\xi\Sigma(x)).
\end{equation}
Moreover, differentiating $|\Sigma(x)|^2=1$ gives
$\langle\nabla_\xi\Sigma(x), \Sigma(x)\rangle =0$, which implies that $\nabla_\xi\Sigma(x)\in (\Sigma(x))^\perp$. Hence recalling \eqref{qq} and \eqref{m0},
\[
|A_x(\nabla_\xi\Sigma(x))|
=
|\nabla_\xi\Sigma(x)|
\ge m_0.
\]
Using $p=a\Sigma(x)+q$ and \eqref{de2}, we obtain
\[
\begin{aligned}
\nabla_\xi h(x)
&=(\nabla_\xi A_x)(p)
=-aA_x(\nabla_\xi\Sigma(x))+(\nabla_\xi A_x)(q).
\end{aligned}
\]
Therefore,
\[
|\nabla_\xi h(x)|
\ge m_0|a|-B_0|q|
=m_0|a|-B_0|h(x)|.
\]
It follows that
\[
\begin{aligned}
|p|
&\le |a|+|q|
\le
\frac{1}{m_0}|\nabla_\xi h(x)|
+\left(1+\frac{B_0}{m_0}\right)|h(x)|.
\end{aligned}
\]
Setting
\[
c_\Sigma:=\frac{m_0}{m_0+B_0+1}>0,
\]
we conclude that
\begin{equation}\label{eq:cinematic-pointwise-lower}
|h(x)|+|\nabla_\xi h(x)|
\ge c_\Sigma|p|
\qquad
\text{for every }x\in [0,1]^{n-2},\ \xi\in S^{n-3},
\end{equation}
which gives \eqref{cine2}.
Thus condition~(3) holds.

\medskip

Finally, we show (2) in Definition \ref{cine}. Indeed, combining \eqref{dou1} and \eqref{cine2}, we deduce that
\begin{equation*}
 |y-z| \lesssim \inf_{x \in [0,1]^{n-2}, \xi \in S^{n-3}} \{ |h(x)| + |\nabla_\xi h(x)| \} \le  \|f_y-f_{z}\|_{C^1([0,1]^{n-2})} \lesssim |y-z|,
\end{equation*}
which implies that (2) in Definition \ref{cine} holds and
$F$ is a doubling space with the doubling constant comparable to $\mathbb{R}^{n}$. In particular, $F$ is biLipschitz homeomorphic to $Z$.

The proof is complete.
\end{proof}

Finally, Lemma \ref{pairvolume2} holds as a direct consequence of Lemma \ref{fg} and Lemma \ref{cineprop}.

\section{Line-cone Incidence Estimate}\label{s4}

We prove Theorem \ref{lineconen} in this section and we recall the statement of Theorem \ref{lineconen} here.

\medskip
\noindent
{\bf Theorem 1.20.}
Let $\Sigma$ be as in Theorem \ref{cormain}.
If a line $L \subset \R^n$ satisfies $\min_{p \in \calC_z}\{\angle(L,T_p\calC_z)\}  \ge a$ for some $a \gg \delta >0$, then
  $$\calL^{n}(\calC_z^\delta \cap L^\delta) \lesssim_a \delta^n $$
  where we recall the cone $\calC_z$ with apex $z$ and cross-section $\Sigma$ is the set
$$\calC_z = \bigcup_{p \in \Sigma} L_{z,p} = \bigcup_{p \in \Sigma, r \in [-1,1]} z + rp \subset \R^n.$$

Before presenting the proof, we make several reductions. First, we can partition $\Sigma$ into small pieces $\Sigma_i$ and show Theorem \ref{lineconen} holds for each cone with cross-section $\Sigma_i$ sufficiently small. Thus in the following of this section,
by partitioning $\Sigma$ into small pieces $\Sigma_i$, we can assume each $\Sigma_i: [0,1]^{n-2} \to S^{n-1}$ has diameter $<\frac{\pi}{100}$. In particular, for any $x,y \in \Sigma_i$, $\angle(x,y) < \frac{\pi}{100}$. Moreover, we require that the smallness of $\Sigma_i$ also guarantees the validity of Lemma \ref{int2} and Lemma \ref{smalla} below.
In the following, we focus on one piece $\Sigma_i$ and write $\Sigma=\Sigma_i$.

Second, noting that
$$ \calL^{n}(\calC_z^\delta \cap L^\delta) = \calL^{n}[(\calC_z^\delta-\{z\}) \cap (L^\delta-\{z\})] = \calL^{n}[\calC_o^\delta \cap (L^\delta-\{z\})], $$
it suffices to show Theorem \ref{lineconen} holds for the cone $\calC_o$ with apex as the origin $o$. In the following, we write $\calC_o = \calC$.

Third, noting that we can decompose $\calC$ as $\calC_+ \cup \calC_-$ where
$$ \calC_\pm:=  \bigcup_{p \in \Sigma, r \in [0,1]}  \pm rp, $$
to show Theorem \ref{lineconen}, it suffices to estimate $\calL^{n}(\calC^\delta_\pm \cap L^\delta)$ separately. Since the proof will be the same for both $\calC_\pm$, we will only focus on $\calC_+$ and write $\calC_+=\calC$ and $L_{o,p}=\cup_{r \in [0,1]}rp$ throughout the rest of this section.
Moreover,
since $\calL^{n}(\calC^\delta \cap L^\delta)$ is invariant under any isometry of $\R^n$, we can always rotate $\calC^\delta$ with respect to $o$ such that one generatrix of $\calC$ is the direction of the first axis $e_1$ and the tangent space of this generatrix is span$\{e_1, \cdots, e_{n-1}\}$ where $\{e_i\}_{i=1,\cdots,n}$ is the standard orthonormal basis of $\R^n$. The choice of this generatrix will be in the beginning of the proof of Theorem \ref{lineconen} below \eqref{x} at the end of this section. In particular, $\Sigma$ passing through $(1,0,\cdots,0)$.
By doing this and up to a flipping with respect to span$\{e_1, \cdots, e_{n-1}\}$, the cone $\calC$ can be seen as a graph of a function $h$ from a domain $U \to \R_+$ where $U=\pi_{e_n^\perp}(\calC) \subset$ span$\{e_1, \cdots, e_{n-1}\}$ is the orthogonal projection $\pi_{e_n^\perp}$ of $\calC$ to span$\{e_1, \cdots, e_{n-1}\}$ (noting that by requiring the diameter of $\Sigma$ sufficiently small we could assume that the projection acting on $\calC$ is a homeomorphism and the function is well-defined on $U$). Since we decompose $\Sigma$ into small pieces, the Euclidean $\delta$-neighbourhood of $\calC$ is comparable to its vertical $\delta$-neighbourhood.

We need the following lemma.
\begin{lemma}\label{int2}
For every line $L\subset\mathbb R^n$, the set
\[
L\cap \calC^\delta
\]
has at most two connected components.
\end{lemma}
Since the proof of the lemma is elementary, we postpone its proof to Appendix A.

\begin{lemma} \label{smalla}
  Let $h: U \to \R$ be the function such that graph$(h)=\calC$. For any $a>0$,
  if $\Sigma$ is a sufficiently small piece depending on $a$, then
  \begin{equation}\label{10a}
    |\nabla h(x)| \le \frac{a}{10}, \quad \forall x \in U \setminus\{o\}.
  \end{equation}
\end{lemma}

\begin{proof}
    By the convention at the beginning of this section, there is a generatrix $L_{o,p}$ of $\calC$ with direction $e_1$ where $p \in \Sigma$. In particular, $L_{o,p} \subset U$. We first show the lemma for points in this generatrix.

  Let $x \in L_{o,p}$. Then by our convention at the beginning of this section, we know $T_x\calC =$ span$\{e_1,\cdots,e_{n-1}\}$. Thus at $x$, graph$(h)=\calC$ is tangent to span$\{e_1,\cdots,e_{n-1}\}$. This implies that $\nabla h(x)=0$. Thus the lemma holds for $x \in L_{o,p}\setminus\{o\}$.

  We then prove the lemma for other points. To this end, consider the following decomposition of $U$. Note that we can write $\calC$ as
     \begin{equation*}
       \calC \setminus \{o\}= \bigcup_{r\in(0,1]} {\bf r}^{-1}(\Sigma) \cap S^{n-1}_r =: \bigcup_{r\in(0,1]} \Sigma_r
     \end{equation*}
  where ${\bf r}^{-1}$ is the preimage set of the radial projection and $S^{n-1}_r$ is the sphere centred at $o$ with radius $r$.
  For each $r \in (0,1]$, let $U_r$ be the orthogonal projection of $\Sigma_r$ to span$\{e_1,\cdots,e_{n-1}\}$. Then, we know
  $$  U = \bigcup_{r \in (0,1]} U_r.$$
    Let $y \in U_r$ for some $r \in (0,1]$.
  Since $\Sigma$ is $C^2$, we know $\nabla h$ is continuously depending on $U\setminus\{o\}$, thus if $\Sigma$ is chosen sufficiently small, we know $U_r$ is also small such that all points $y \in U_r$ are sufficiently close to $U_r \cap L_{o,p}$.
  As a result, $|\nabla h(y)| \le \frac{a}{10}$ for all $y \in U_r$.

  The proof is complete.
\end{proof}

\begin{remark}
  Since $\calC=$ graph$(h)$, the vertical $\delta$-neighbourhood of $\calC$ is exactly $h^\delta=\{ (x,y)\in U \times \R: |y-h(x)|< \delta \}$. To show Theorem \ref{lineconen}, it is convenient to work with $h^\delta$. However, since the Euclidean $\delta$-neighbourhood $\calC^\delta$ of $\calC$ is not contained in $h^\delta$. We need make the following Lipschitz extension of $h$.

  Let $\bar U$ be the the orthogonal projection $\pi_{e_n^\perp}(\calC^\delta)$ to $e_n^\perp$. Recalling that $U=\pi_{e_n^\perp}(\calC)$, we know $\bar U \subset U^\delta$ as subsets in $\R^{n-1}$. Let $\tilde h: U^\delta \to \R$ be the Macshane extension of $h$, that is,
  $$ \tilde h(x):= \inf_{u \in U}\{ h(u) + m|x-u|\}, \quad x \in U^\delta $$
  where $m = \max_{x \in U} |\nabla h(x)| \le \frac{a}{10}$ by Lemma \ref{smalla}. Then $\tilde h =h$ in $U$ and $\max_{x \in U} |\tilde \nabla h(x)| \le m \le \frac{a}{10}$. In particular, $\calC^\delta \subset \tilde h^{c(a)\delta}$ for $c(a)>0$ depending on $a$. To show Theorem \ref{lineconen}, it suffices to estimate $\calL^n(L^\delta \cap \tilde h^{c(a)\delta})$. In the following proof, by abuse of notation, we write the $h:  U^\delta \to \R$ instead of $\tilde h$ as the above extension.
\end{remark}

We are able to show Theorem \ref{lineconen}.

\begin{proof}[Proof of Theorem \ref{lineconen}]
    Recall
    $$ a:= \min_{p \in C\setminus\{o\}}\{ \angle(L,T_pC) \}. $$
    By
     Lemma \ref{int2} and the above remark,
     It suffices to show each interval $I \subset L$ in the intersection $L \cap h^\delta$ has length upper bound $\lesssim \frac{\delta}{a}$.
     Thus,
     we concentrate on a fixed interval $I \subset L$.

    We consider the following two cases depending on the Euclidean distance $d(o,I)$ from $o$ to $I$.

    \medskip
    \emph{Case 1. $I \cap B(o,10 \delta) = \emptyset$.}

    Since $I \subset L \cap h^\delta$, we can find $x\in I \subset L \cap h^\delta$ and $y \in \calC$ such that
    $$  x = y + \eta \nu_y  $$
    where $\nu_y$ is the unit normal of $T_y\calC$ and $\eta \in (-\delta, \delta)$.
    Equivalently, we have
    \begin{equation}\label{x}
       |x-y| < \delta, \mbox{ and } x-y \perp T_y\calC.
    \end{equation}
    In particular, since $x \notin B(o,10 \delta)$, $|y| > 9 \delta$ and hence $y \ne o$.

    As mentioned in the beginning of this section, we choose $y$ as the first coordinate and $y = |y|e_1$ and $T_y\calC=$ span$\{e_1, \cdots, e_{n-1}\}$. Thus $\nu_y= e_n$ and $x = y+ (x \cdot e_n)e_n$.
    And we view $\calC$ as the graph of a function $h:U^\delta \to \R$ where we recall $U \subset$ span$\{e_1,\cdots,e_{n-1}\}$ is the orthogonal projection of $\calC$ to span$\{e_1,\cdots,e_{n-1}\}$.
    Let $y_\ast:= \frac{y}{|y|} \in \Sigma$. We note that the generatrix $L_{o,y_\ast} \subset$ span$\{e_1\}\cap \calC \subset U$. In particular, $y \in U^\delta$. The following figure shows the case when $n=3$ and $L$ passing through $\calC$ (where one has $x=y$).

    \begin{figure}[h!]
\begin{center}
\begin{overpic}[scale = 0.5]{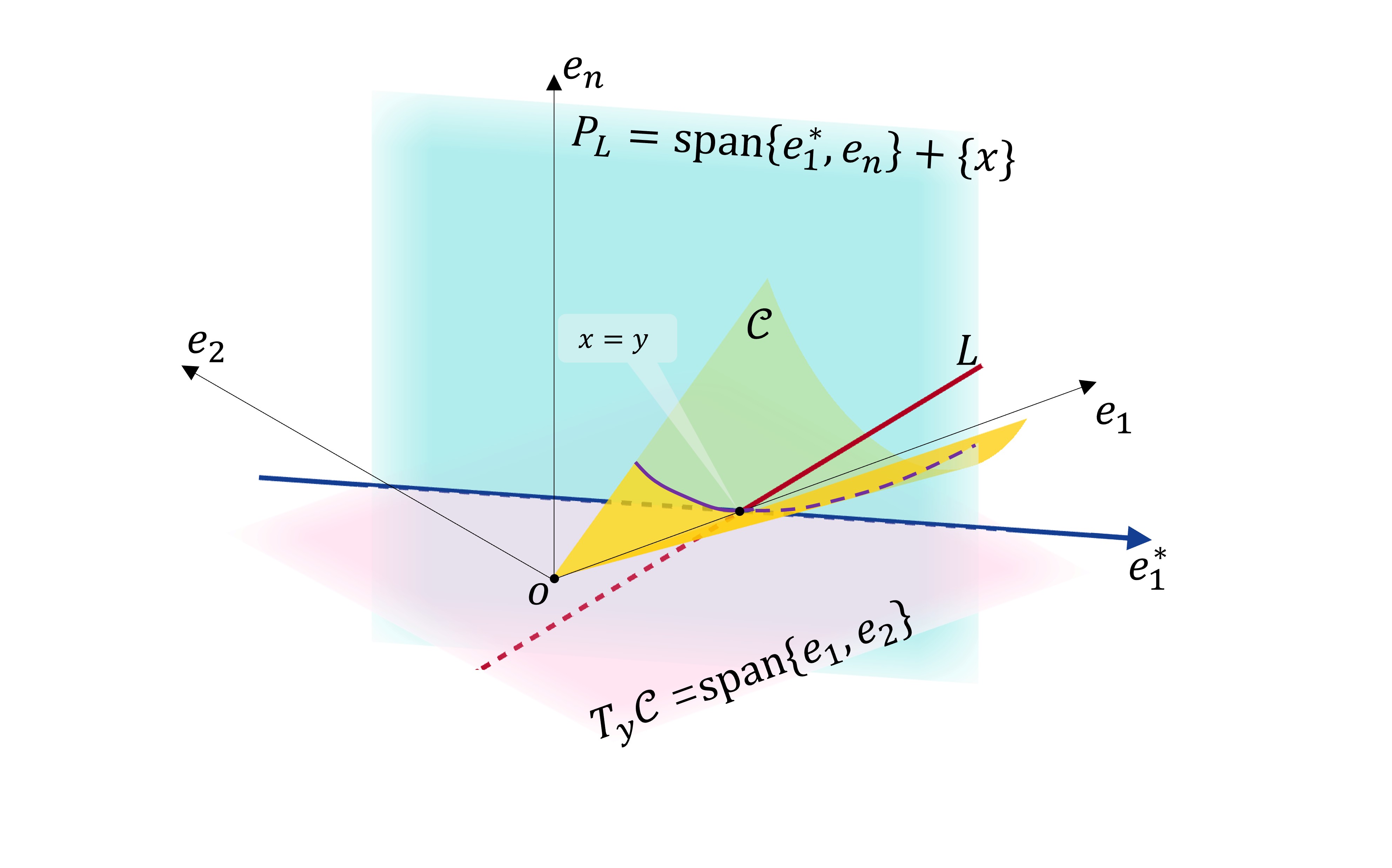}
\end{overpic}
\caption{Cone $\calC$, line $L$, direction $e_1^\ast$ and plane $P_L$}\label{fig1}
\end{center}
\end{figure}

\medskip
    \emph{Subcase 1-1. $L \perp T_y\calC$.}

    Namely, $L$ is parallel to $e_n$ (i.e. $\angle (L,T_y\calC) =\frac{\pi}2$). In this subcase, it is easy to see $|L \cap h^\delta|= 2\delta$ since the vertical thickness of $h^\delta$ is $2\delta$.

\medskip
    \emph{Subcase 1-2. $\angle (L,T_y\calC) \in [a, \frac{\pi}{2})$.} We refer to Figure \ref{fig2} for an illustration.

        \begin{figure}[h!]
\begin{center}
\begin{overpic}[scale = 0.6]{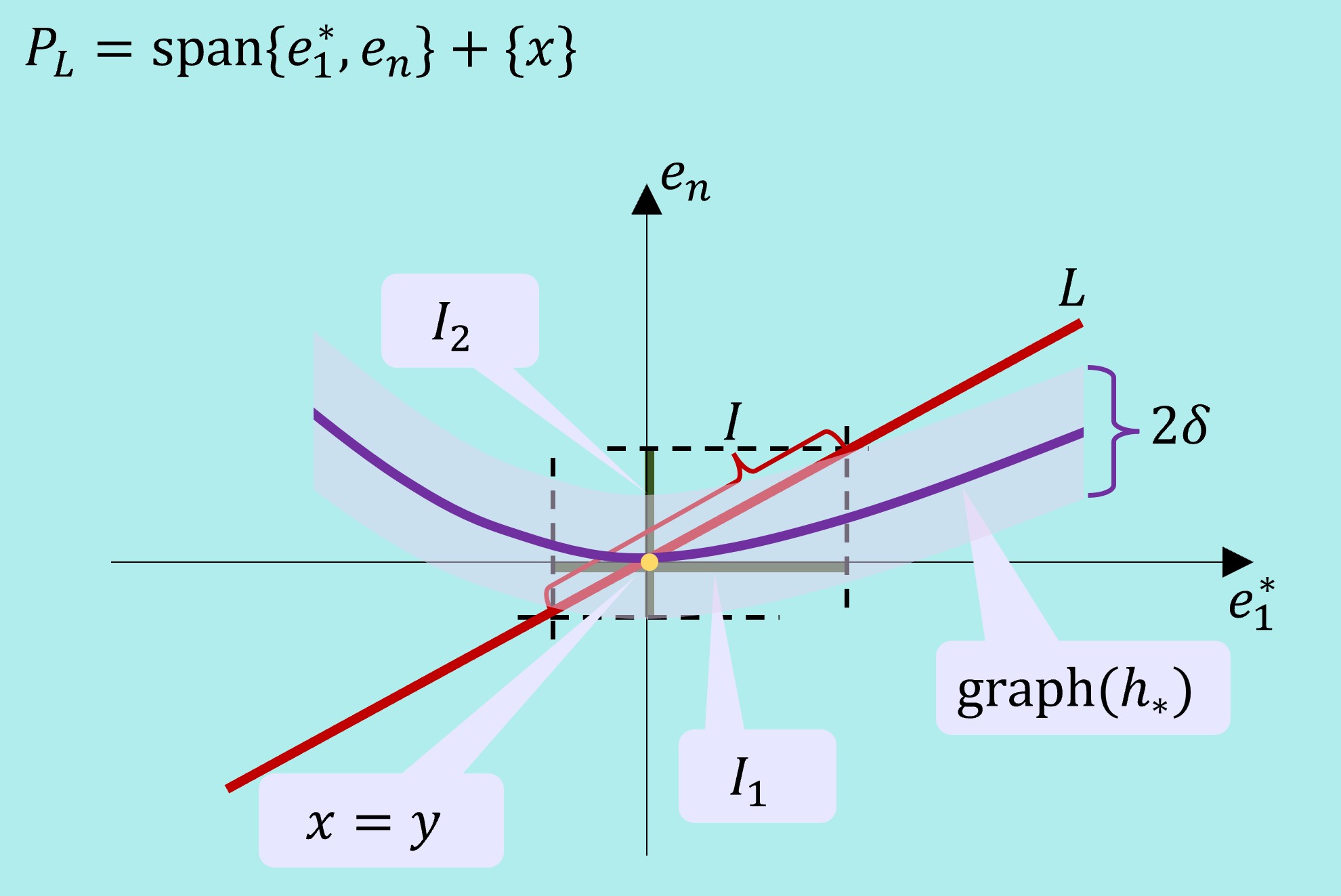}
\end{overpic}
\caption{Graph$(h_\ast)$, intervals $I$, $I_1$ and $I_2$}\label{fig2}
\end{center}
\end{figure}

     Let $e_1^\ast$ be the direction of the orthogonal projection of $L$ into $T_y\calC=$ span$\{e_1, \cdots, e_{n-1}\}$. Since $x \in L$, we know
     $$L \subset P_L:=\mbox{span}\{e_1^\ast,e_n\}+\{x\}.$$
     By $x-y \perp T_y\calC$ (recalling \eqref{x}),
     we know $x-y$ is parallel to $e_n \in P_L$, thus
     $y \in P_L$.
     In the following, we will focus on the plane $P_L$ and it is convenient to let $y$ be the origin of $P_L$, $e_1^\ast$ be the first coordinate and $e_n$ be the second (vertical) coordinate. Thus $y=(0,0) \in P_L$ and $x$ has coordinate
     $(0,x-y)\in P_L$.

     Define
     $$U_\ast := U^\delta \cap [\mbox{span}\{e_1^\ast\}+ \{y\}].$$
     Then $L \cap (U_\ast \times \mbox{span}\{e_n\})$ can be seen as a graph in $P_L$ of a linear function $g: U_\ast \to \R$. Let $h_\ast:= h|_{U_\ast}$.
     We write $s$ as the elements of $U_\ast$ when considering $U_\ast$ as a subset of $\mbox{span}\{e_1^\ast\}+ \{x\} \subset P_L$ and write $se_{1}^\ast$ as the elements of $U_\ast$ when considering $U_\ast$ as a subset of $\R^n$. Note that
     $s=0$ is the point $y$.

     To estimate the length of the interval $I \subset L$,
     let $I_1$ be the orthogonal projection of $I$ to $\mbox{span}\{e_1^\ast\}+ \{x\}$ and $I_2$ be the orthogonal projection of $I$ to $\mbox{span}\{e_n\}$. We have the relation
     \begin{equation}\label{I12}
        |I| = \frac{|I_1|}{\cos \angle(I,I_1)} = \frac{|I_1|}{\cos \angle(L,T_y\calC )}, \quad
     |I| = \frac{|I_2|}{\cos \angle(I,I_2)} = \frac{|I_2|}{\sin \angle(L,T_y\calC)}
     \end{equation}
     and
     $$ I_1 = \{s \in U_\ast: |g(s)-h_\ast(s)| < \delta\}. $$
     We first estimate the length $|I_1|$.
     Observe that
     $$  h_\ast(0) =  h_\ast(y) = 0 \mbox{ and } |g(0)|=|x-y| < \delta. $$
     Moreover, thanks to Lemma \ref{smalla} and the assumption of this subcase, one has
     \begin{equation}\label{h'}
       |h_\ast'(s)|= |\nabla h(se_1^\ast) \cdot e_1^\ast| \le \frac{a}{10} \mbox{ and } g'(s)= \pm\tan \angle (L,T_y\calC) \ge a, \quad  \forall s\in U_\ast.
     \end{equation}
     (In the following, we assume $g'(s)= \tan \angle (L,T_y\calC)$ and the case $g'(s)= -\tan \angle (L,T_y\calC)$ is similar.)
     Thus
     $$  (g-h_\ast)(s) = (g-h_\ast)(0) + \int_0^s (g'-h_\ast')(\tau) \, d\tau > -\delta + \int_0^s \frac{a}{2} \, d\tau > -\delta + \frac{a}{2}s, \quad s>0. $$
     Thus if $s> \frac{4}{a}\delta$, then $g(s)>h_\ast(s)+\delta$, which implies $(se_1^\ast,g(s)) \notin h^\delta$. Similarly one can show that if $s< -\frac{4}{a}\delta$, then $g(s)< h_\ast(s)-\delta$.
     Thus the interval $I_1$ has length $\lesssim_a \delta$.

     We then estimate $|I_2|$.
     Note that for each point $q \in I_2$, we have
     $$  q \in (h_\ast(s_q)-\delta, h_\ast(s_q)+\delta) \mbox{ for some $s_q\in I_1$}.$$
     Thus
     $$  |I_2| \le |h_\ast(I_1)|+2\delta.   $$
      Thanks to \eqref{h'}, we deduce that
      $$  |I_2| \le \max_{s\in I_1}\{|h'_\ast(s)|\}|I_1|+2\delta < \frac{a}{10}|I_1|+2\delta \lesssim_a \delta.   $$
      Finally, since in this subcase $\sin \angle(L,T_y\calC) \sim \angle(L,T_y\calC) \gtrsim a$, thanks to the second inequality in \eqref{I12}, we conclude that
      $$  |I| \sim_a |I_2| \lesssim_a \delta   $$
      as desired.

%

\medskip
    \emph{Case 2. $I \cap B(o,10 \delta) \ne \emptyset$.}
In this case, since the interval $I \subset L$, we let
$$ I_1':= I \cap B(o,10 \delta) \mbox{ and } I_2':= I \setminus B(o,10 \delta).  $$
Thus
$$  |I| = |I_1'| + |I_2'|. $$
We immediately have $|I_1'| \le \diam B(o,10 \delta) \le 20 \delta.$
In a similar way as in Case 1, one has $|I_2'| \lesssim_a \delta$, which then implies that
$|I| \lesssim_a \delta$ as desired.
The proof is complete.
\end{proof}

\section{Proof of Theorem \ref{cormain}} \label{s5}
In this section, we will first prove Lemma \ref{volume2} and Theorem \ref{cormain}.

Recall the cinematic maps $f_z$ from \eqref{fw},
$$f_z(x):=(E_1(x)\cdot z, \cdots, E_{n-1}(x)\cdot z) \quad \forall x\in [0,1]^{n-2}$$
and the cone with apex $z$ and cross-section $\Sigma$
$$ \calC_z = \bigcup_{x \in \Sigma,r \in [-1,1]} z + rx = \bigcup_{x \in \Sigma,r \in [-1,1]} L_{z,x}. $$
Note that if $y \in L_{w,x}$, then $L_{y,x}$ and $L_{w,x}$ are contained in the same line in $\R^n$.

\begin{lemma} \label{pairs}
    Let $Z \subset B(o,1/2) \subset \R^n$ and $z \in Z$. For any $z' \in Z \setminus\{z\}$,
    if $f_z^\delta \cap f_{z'}^\delta \ne \emptyset$, then $z' \in \calC_z^{2\delta}$.
\end{lemma}

\begin{proof}
   Since $f_z^\delta \cap f_{z'}^\delta \ne \emptyset$, we know that there exists $x \in [0,1]^{n-2}$ such that
   $$ |f_z(x)-f_{z'}(x)| < 2\delta. $$
   This implies that the orthogonal projection $\pi_{P_x}(z)$ and $\pi_{P_{x}}(z')$ of $z$ and of $z'$ to the hyperplane $P_{x}= T_{\Sigma(x)}S^{n-1}$
   satisfies
   \begin{equation}\label{zz'}
     |\pi_{P_{x}}(z)-\pi_{P_{x}}(z')| < 2\delta.
   \end{equation}
   Recalling that $\Sigma(x)$ is the unit normal of $P_{x}$ and $z' \in Z \subset B(o,1/2)$, we have
   $$  z' = \pi_{P_{x}}(z') + r_{z'}\Sigma(x) \mbox{ for some $r_{z'} \in [-1/2,1/2]$}.  $$
   Combining \eqref{zz'}, we deduce that
   \begin{equation}\label{z'z}
     |z' - (\pi_{P_{x}}(z) + r_{z'}\Sigma(x))|  < 2\delta.
   \end{equation}
   Moreover,
   by $z \subset Z \subset B(o,1/2)$, we have $|\pi_{P_{x}}(z)-z| < \frac12$ and hence
   $$  \pi_{P_{x}}(z) + r_{z'}\Sigma(x) \in  \bigcup_{r \in [-1,1]} z + r\Sigma(x) = L_{z,\Sigma(x)} \subset \calC_z.  $$
   Combining \eqref{z'z}, we conclude
   $$   z' \in \calC_z^{2 \delta}$$
   as desired.
\end{proof}

\begin{lemma} \label{Lambda}
    Under the assumption on $\Sigma$ in Theorem \ref{cormain}, for any $y \in S^{n-1}$, define
    $$  \Lambda_y:=  \{x \in \Sigma : \mbox{\rm span}\{y\} \subset T_x\calC_o\}.  $$
    Then
    $$  \dim \Lambda_y \le n-3.  $$
\end{lemma}

To show Lemma \ref{Lambda}, we need the following geometric fact.
\begin{remark} \label{a}
   Assume $n \ge 4$ and $\Sigma$ satisfies Theorem \ref{cormain}.
  Since $\Sigma \subset S^{n-1}$, Gauss' equation reads as
   $$ K_{\xi_i,\xi_j}(x) = \kappa_{i}(x) \kappa_{j}(x) + 1 ,\quad x \in \Sigma  $$
  where $\{\xi_i\}_{1 \le i\le n-2} \subset T_x\Sigma$ are the $n-2$ principal directions,
  $K_{\xi_i,\xi_j}(x)$ is the sectional curvature of $\Sigma$ at $x$ spanned by $\xi_i$ and $\xi_j$ and $\kappa_{i}(x)$ is the principal curvature in the direction $\xi_i$. Thus the sectional curvature $>1$ is equivalent to that the $n-2$ principal curvatures at each $x \in \Sigma$, $\kappa_1(x), \cdots, \kappa_{n-2}(x)$ are non-vanishing and have same sign. Fixing an orientation of the normal vector $\nu$, we assume $\kappa_i(x)>0$ for all $x \in \Sigma$ and $i=1, \cdots, n-2$.

Recall that the second fundamental form $II(x)$ at $x \in \Sigma$ is the following symmetric bilinear form
$$ II(x): T_x\Sigma \times T_x\Sigma \to \R, \quad II(x)(\xi,\zeta):= \langle\nabla_\xi \zeta,\nu(x)\rangle.$$
The eigenvalues of $II(x)$ are precisely the $n-2$ principal curvatures at each $x \in \Sigma$. Thus, $II(x)$ is positive definite at each $x \in \Sigma$.
\end{remark}

\begin{proof}[Proof of Lemma \ref{Lambda}]
We only prove the case $n\ge 4$. The case $n=3$ can be shown similarly.
    Since $T_x\calC_o = T_x\Sigma \, \oplus $ span$\{x\}$ and
    $\R^n=\mbox{span}\{\nu(x)\}\oplus\mbox{span}\{x\}\oplus T_x\Sigma$ where we recall $\nu$ is the unit normal of $\Sigma$ at $x$, we know $(T_x\calC_o)^\perp = $ span$\{\nu(x)\}$.
    We deduce that $x \in \Lambda_y$ is equivalent to $\nu(x) \perp y$.
    Define
    $$ G: \Sigma \to \R, \quad G(x):=\nu(x)\cdot y.$$
    Then $\Lambda_y \subset G^{-1}(0)$.
    For any $x \in \Lambda_y$, that is, $\mbox{\rm span}\{y\} \subset T_x\calC_o= T_x\Sigma \oplus \mbox{\rm span}\{x\} $,
    write $y=a_x x+ v_x$ where $a_x \in\R$ and $v_x \in T_x\Sigma$.
    Then for any $0 \ne \xi \in T_x\Sigma$,
    $$ \nabla G(x) \cdot  \xi = -II(x)(\xi,v_x).$$
    Thanks to Remark \ref{a}, $II(x)$ is positive definite. We know $\nabla G(x) =0 $ only if $v_x=0$, that is, $y=\pm x$. Hence by the inverse function theorem, for all $x \in \Lambda_y\setminus \{\pm y\}$, $\Lambda_y = G^{-1}(0)$ is a submanifold of dimension $n-3$.
%
%
\end{proof}

In the following proofs, we need the notion of $(\delta,s)$-set.

\begin{definition}[$(\delta,s)$-set] \label{dst}
Let $s \geq 0$, $C > 0$, $\delta> 0$ and $(E,d)$ be a metric space. A bounded set $A \subset E$ is called a \emph{$(\delta,s,C)$-set} if
\begin{displaymath}
|A \cap B(x,r)|_{\delta} \leq Cr^{s}|A|_{\delta}, \qquad x \in E, \, r \geq \delta. \end{displaymath}
  \end{definition}

\begin{remark} \label{deltaset}
  \rm We need following properties of $(\delta,s,C)$-set.

  (i) If $A$ is a non-empty $(\delta,s,C)$-set, then $|A|_{\delta} \geq C^{-1}\delta^{-s}$. This follows by applying the defining condition at scale $r = \delta$.

  (ii) If $A$ is a non-empty $(\delta,s,C)$-set and $A' \subset A$ satisfies $|A'|_{\delta} \ge D^{-1}|A|_{\delta}$ for some $D \ge 1$, then
  \begin{displaymath}
  |A' \cap B(x,r)|_{\delta} \leq |A \cap B(x,r)|_{\delta} \le Cr^{s}|A|_{\delta} \le CDr^{s}|A'|_{\delta}, \qquad x \in E, \, r \geq \delta, \end{displaymath}
  which implies that $A'$ is a $(\delta,s,CD)$-set.
\end{remark}

\begin{lemma} \label{disw}
Let \(n\ge3\) and \(Z\subset B(o,1/2)\subset\R^n\) be analytic with
\(t:=\dim Z>n-2\). Suppose that \(y\in S^{n-1}\) satisfies \eqref{le1}, i.e.
\[
 E:=\pi_{T_yS^{n-1}}(Z),\quad d:=\dim E\le n-2.
\]
For every \(\epsilon>0\) and \(1<t'<t\), there exists
\(s\in(0,d)\) and $\delta_0(\epsilon,t)$ such that for all \(\delta<\delta_0\),
there exists $\delta$-separated \((\delta,s,\delta^{-\epsilon})\)-set $W \subset E$ with $|W| \le \delta^{-s}$ and $\delta$-separated \((\delta,t',\delta^{-5\epsilon})\)-set $Y \subset Z$ with $|Y|\ge \delta^{-t'+5\epsilon}$ satisfying
\begin{equation}\label{eq:58}
 Y\subset\pi_{T_yS^{n-1}}^{-1}(W^{n\delta})
   =\bigcup_{w\in W}\pi_{T_yS^{n-1}}^{-1}\bigl(B(w,n\delta)\bigr)
\end{equation}
and
\begin{equation}\label{eq:59}
 Y\subset\bigcup_{w\in W}L_{w,y}^{n\delta},
 \qquad L_{w,y}:=\bigcup_{r\in[-1,1]}\{w+ry\}.
\end{equation}
\end{lemma}

Since the proof of Lemma \ref{disw} is technical, we postpone its proof to the end.
We are ready to show Lemma \ref{volume2} and Theorem \ref{cormain}.

\begin{proof}[Proof of Lemma \ref{volume2} and Theorem \ref{cormain}]
    Recall the point $y \in S^{n-1}$ in Theorem \ref{cormain}.
    It suffices to show that Theorem \ref{cormain} holds for $\Sigma \setminus (\cup_{x \in \Lambda_y}B(x,a))$ for each $a>0$ small enough where $\Lambda_y$ is the set in Lemma \ref{Lambda}.
    Thus we fix $a>0$ small enough and partition $\Sigma \setminus (\cup_{x \in \Lambda_y}B(x,a))$ into small pieces such that all the results in Section \ref{s4} hold.
    We prove Theorem \ref{cormain} for each piece. Fix a piece and still denote it by $\Sigma$. In particular, for each $z \in \R^n$ and $p \in \calC_z\setminus\{z\}$, noting that $\frac{p-z}{|p-z|} \in \Sigma$, we have
    \begin{equation}\label{ang2}
      \angle (y, T_p\calC_z)=\angle (y, T_{p-z}\calC_{o})=\angle (y, T_{\frac{p-z}{|p-z|}}\calC_o) \ge c(a)>0
    \end{equation}
    where $c(a)$ is a constant only depending on $\Sigma$.

    We first prove Lemma \ref{volume2}. By Lemma \ref{cineprop}, let $D>1$ be the bilipschitz constant of the bilipschitz map from $Z$ to $F=\{f_{z} \in C^2( [0,1]^{n-2},[0,1]^{n-1})\}_{z \in Z}$. For $\epsilon>0$, we choose $\delta_0=\delta_0(\epsilon)>0$ such that
    $$  \delta_0^\epsilon < \min\{D^{-(n-2)}, \log(1/\delta_0)^{-1}\}.  $$
    We will show Lemma \ref{volume2} holds for all $\delta<\delta_0$.
     Recalling from \eqref{heu31} that
     $$\mathcal{Z} = \bigcup_{z \in Z} \bigcup_{x \in X_z}(x, f_z(x)) \subset [0,1]^{2n-3}$$
     where
     $X_z \subset [0,1]^{n-2}$ with $\calL^{n-2}(X_z)>C_Z>0$ for each $z \in Z$.
    To show
    $ |\mathcal{Z}|_{\delta} \gtrsim \delta^{9\epsilon - (n-2 +t)}$ in Lemma \ref{volume2},
    it suffices to show for any $0<t' < t$, there exists $Y \subset Z$ such that
    \begin{equation}\label{caly}
      |\mathcal{Y}|_{\delta}\gtrsim \delta^{9\epsilon - (n-2 +t')}
    \end{equation}
    where
     $$  \mathcal{Y} := \bigcup_{z \in Y} \bigcup_{x \in X_z }(x, f_z(x)). $$
The following proof is devoted to showing \eqref{caly}.

By Lemma \ref{disw}, for each $t'<t$, we can find a \((\delta,s,\delta^{-\epsilon})\)-set $W \subset \pi_{T_yS^{n-1}}(Z)$ with $|W| \le \delta^{-s}$ and \((\delta,t',\delta^{-5\epsilon})\)-set $Y \subset Z$ with $|Y|\ge \delta^{-t'+5\epsilon}$ satisfying \eqref{eq:58} and \eqref{eq:59}.

Noting that $\{L_{w,y}\}_{w \in W}$ are parallel line segments perpendicular to $T_yS^{n-1}$, we have
$$ \pi_{T_yS^{n-1}}(L_{w,y}) = w \mbox{ and } \pi_{T_yS^{n-1}}(L_{w,y}^{n\delta}) = B^{n-1}(w,n\delta) \subset T_yS^{n-1}. $$
Thus for any ball $B^n(x,r) \subset \R^n$ with $r \in [\delta,1]$, we have
\begin{align*}
  |\{w \in W: L_{w,y}^{n\delta} \cap B^n(x,r) \ne \emptyset\}| & = |\{w \in W: \pi_{T_yS^{n-1}}[L_{w,y}^{n\delta} \cap B^n(x,r)] \ne \emptyset\}| \\
   & =|\{w \in W: B^{n-1}(w,n\delta) \cap B^{n-1}(\pi_{T_yS^{n-1}}(x),r) \ne \emptyset\}| \\
   & \lesssim |W^{n\delta} \cap B^{n-1}(\pi_{T_yS^{n-1}}(x),r)|_{\delta} \\
   & \lesssim |W \cap B^{n-1}(\pi_{T_yS^{n-1}}(x),r)|_{\delta} \le \delta^{-\epsilon}r^s|W|
\end{align*}
where in the last inequality we recall that $W$ is a $(\delta,s,\delta^{-\epsilon})$-set.

For each $z \in Y$, define
$$  Y_z : = \calC_z^{n\delta} \cap (Y\setminus \{z\}).  $$
For each $r \in [\delta,1]$, thanks to \eqref{eq:59}, recalling \eqref{ang2} and applying Theorem \ref{lineconen}, we obtain that
\begin{equation}\label{ds}
    |Y_z \cap B(x,r)|  \le  |\{w \in W: L_{w,y}^{n\delta} \cap B(x,r) \ne \emptyset\}| \cdot \max_{ w \in W} |L_{w,y}^{n\delta} \cap \calC_z^{n\delta}|_\delta \lesssim_a \delta^{-\epsilon}r^s|W|.
\end{equation}
   For each dyadic $r \in [\delta,1]$ and $z \in Y$, define
   $$  Y_{z,r}:= \{ z' \in Y_z : |z'-z|  \in [r,2r)  \} . $$
   We deduce from \eqref{ds} that
   \begin{equation}\label{yzr}
     |Y_{z,r}| \le |Y_z \cap B(z,2r)| \lesssim_a \delta^{-\epsilon}r^s|W| \le \delta^{-\epsilon}r^s\delta^{-s} .
   \end{equation}
   Moreover, recall from Lemma \ref{cineprop} that
  $F=\{f_{z} \in C^2( [0,1]^{n-2},[0,1]^{n-1})\}_{z \in Z}$ is a cinematic family and $F$ is biLipschitz homeomorphic to $Z$. Applying Lemma \ref{fg}, we have
  \begin{equation}\label{bilip}
    \mathcal{L}^{2n-3}( f^\delta_z(X_z^\delta)\cap  f^\delta_{z'}(X_{z'}^\delta)) \lesssim \frac{\delta^{2n-3}}{\|f_z-f_{z'}\|_{C^2([0,1]^{n-2})}^{n-2}} \le D^{n-2} \frac{\delta^{2n-3}}{|z-z'|^{n-2}} \le \frac{\delta^{-\epsilon}\delta^{2n-3}}{|z-z'|^{n-2}}.
  \end{equation}
  Combining \eqref{yzr} and \eqref{bilip}, we arrive at
    \begin{align*}
\sum_{\substack{z,z'\in Y\\z\ne z'}}\mathcal{L}^{2n-3}( f^\delta_z(X_z^\delta)\cap  f^\delta_{z'}(X_{z'}^\delta)) &= \sum_{z\in Y}\sum_{r=\delta}^1\sum_{z'\in Y_{z,r}} \mathcal{L}^{2n-3}( f^\delta_z(X_z^\delta)\cap  f^\delta_{z'}(X_{z'}^\delta))  \\
& \lesssim \sum_{z\in Y}\sum_{r=\delta}^1 \sum_{z'\in Y_{z,r}} \frac{\delta^{2n-3}}{\|f_z-f_{z'}\|_{C^2([0,1]^{n-2})}^{n-2}}
 \le \sum_{z\in Y}\sum_{r=\delta}^1 \sum_{z'\in Y_{z,r}} \frac{\delta^{-\epsilon}\delta^{2n-3}}{|z-z'|^{n-2}} \\
& \le \sum_{z\in Y}\sum_{r=\delta}^1 |Y_{z,r}|\frac{\delta^{-\epsilon}\delta^{2n-3}}{r^{n-2}}
 \le \sum_{z\in Y}\sum_{r=\delta}^1 \delta^{-2\epsilon} r^{s}\delta^{-s}\frac{\delta^{2n-3}}{r^{n-2}} \\
 & = |Y|\log(1/\delta)\delta^{-2\epsilon} \frac{\delta^{2n-3-s}}{\delta^{n-2-s}} \le \delta^{-3\epsilon}
  \delta^{n-1}|Y|
  \end{align*}
where in the last line we recall $s < \dim \pi_{T_yS^{n-1}}(Z) \le n-2$ from the assumption \eqref{le1}.
Thus
    \begin{align*}
    \sum_{z,z'\in Y}\mathcal{L}^{2n-3}( f^\delta_z(X_z^\delta)\cap  f^\delta_{z'}(X_{z'}^\delta))&=
\sum_{\substack{z,z'\in Y\\z\ne z'}}\mathcal{L}^{2n-3}( f^\delta_z(X_z^\delta)\cap  f^\delta_{z'}(X_{z'}^\delta))
 + \sum_{z\in Y}\mathcal{L}^{2n-3}( f^\delta_z(X_z^\delta)) \\
 & \lesssim \delta^{-3\epsilon}
  \delta^{n-1}|Y| + \delta^{n-1}|Y| \le \delta^{-4\epsilon}
  \delta^{n-1}|Y|.
  \end{align*}
Using Cauchy-Schwarz type inequality for Lebesgue measure (c.f. \cite[Lemma4.2]{LIU2025}), we conclude that
 \begin{align*}
    \mathcal{L}^{2n-3}(\bigcup_{z \in Y}  f^\delta_z(X_z^\delta))  &\ge \frac{[\sum_{z \in Y}\mathcal{L}^{2n-3}( f^\delta_z(X_z^\delta)]^2}{\sum_{z,z'\in Y}\mathcal{L}^{2n-3}( f^\delta_z(X_z^\delta)\cap  f^\delta_{z'}(X_{z'}^\delta))}  \\
  & \gtrsim \frac{|Y|^2 \delta^{2n-2}}{\delta^{-4\epsilon} \delta^{n-1}|Y|} = \delta^{4\epsilon} \delta^{n-1}|Y|.
  \end{align*}
  Recalling $|Y| \gtrsim \delta^{-t+5\epsilon}$ and $\delta$ ball in $\R^{2n-3}$ has volume $\sim \delta^{2n-3}$, we obtain \eqref{caly}
as desired. The proof of Lemma \ref{volume2} is complete.

Finally, Theorem \ref{cormain} can be shown in the same way as the proof of \cite[Section 3.2]{orponen2023hausdorff} or \cite[Theorem 1.9]{LIU2025} with the help of Lemma \ref{volume2}. We omit the details.
\end{proof}

We are left to show Lemma \ref{disw}. In the following, fix $y \in S^{n-1}$ and a compact analytic set $Z \subset \R^n$. We always write
\[
 E:=\pi_{T_yS^{n-1}}(Z),\qquad
 d:=\dim\pi_{T_yS^{n-1}}(Z).
\]

We need an auxiliary lemma.

\begin{lemma}\label{frost}
   Let $Z \subset \R^n$ be analytic. Then for any $\ez>0$ sufficiently small and $1<\tau < \dim Z$, there exists a $\tau$-Frostman measure supported on $Z$ such that $\mu_y:=(\pi_{T_yS^{n-1}})_{\#}\mu$ is a $q$-Frostman measure and $\dim F -q < \epsilon$ where $F:=$ spt$\, (\mu_y)$.
\end{lemma}

\begin{proof}
By Frostman's lemma, there exists a probability measure \(\mu_0\) supported
on \(Z\) such that
\[
 \mu_0(B(z,r))\le C_0r^{\tau}
 \qquad(z\in\mathbb R^n,\ r>0).
\]
Write $\mu_{0,y}:=\bigl(\pi_{T_yS^{n-1}}\bigr)_{\#}\mu_0$. However, $\mu_{0,y}$ may not be a $q$-Frostman measure. In the following, we will modify $\mu_{0,y}$ and $\mu_0$ to obtain the desired measures.

Define
\(a_j:=\tau-1+\frac{j\epsilon}{2}\) for integers \(j\ge0\). Let \(j\) be the
least integer for which there is a Borel set \(F\subset E\) satisfying
\[
 \mu_{0,y}(F)>0,
 \qquad
 \mathcal H^{a_j}(F)=0.
\]
Such an integer exists, since \(\mathcal H^{a_j}(E)=0\) whenever
\(a_j>d\). The preceding Frostman estimate implies \(j\ge1\).
Set \(q:=a_{j-1}\). By minimality,
\begin{equation}\label{imp1}
  \mathcal H^q(A)=0
 \quad\Longrightarrow\quad
 \mu_{0,y}(A)=0
 \qquad\text{for every Borel set }A\subset E.
\end{equation}
Moreover, $\mathcal H^{a_j}(F)=0$ implies that $\dim F \le a_j$. 
 \begin{equation}\label{qlower}
   \dim F - q  \le a_j -q = a_j-a_{j-1} <\epsilon .
 \end{equation}

For each integer \(k\ge0\), let \(\mathcal D_k\) be the half-open
cubes in \(T_yS^{n-1}=\mathbb R^{n-1}\) of side length \(2^{-k}\) in a fixed dyadic grid. For
\(M\in\mathbb N\), define
\[
 G_M:=
 \bigcup_{k=0}^{\infty}
 \ \bigcup_{\substack{Q\in\mathcal D_k\\
              \mu_{0,y}(Q)>M(\diam Q)^q}}Q.
\]
The definition includes all dyadic generations with side length
at most \(1\) and every cube in the union
is contained in such a maximal cube. Let \(\mathcal B_M\) be
the inclusion-maximal cubes in this union, that is,
$$\mathcal B_M:= \{Q \in \bigcup_{k=0}^{\infty}\mathcal D_k: Q\subset G_M, \mbox{ and } \mu_{0,y}(Q')\le M(\diam Q')^q, \ \forall Q' \supset Q\}.$$
The elements of
\(\mathcal B_M\) are pairwise disjoint, and therefore
\[
 G_M=\bigcup_{Q\in\mathcal B_M}Q,
 \qquad
 \sum_{Q\in\mathcal B_M}(\diam Q)^q
 \le\frac1M\sum_{Q\in\mathcal B_M}\mu_{0,y}(Q)
 \le\frac1M.
\]
Consequently,
\[
 \mathcal H^q_\infty(G_M)
 \le M^{-1}.
\]
Moreover, noticing that for each $M \ge 1$, $\bigcup_{Q \in \mathcal B_M}Q$ is a cover of $\bigcap_{M=1}^{\infty}G_M$ by cubes of diameter $<M^{-1/q}$,
it follows that
\[
 \mathcal H^q\!\left(\bigcap_{M=1}^{\infty}G_M\right)=0.
\]
Since \(\mu_{0,y}(E)=1\), \eqref{imp1} yields
\(\mu_{0,y}(\bigcap_MG_M)=0\). The sets \(G_M\) are decreasing,
so continuity from above gives \(\limsup_{M \to \infty}\mu_{0,y}(G_M)=0\).

Choose an integer \(M_0\) such that
\[
 \mu_{0,y}(G_{M_0})<\tfrac12\mu_{0,y}(F).
\]
By inner regularity, there is a compact set
\[
 F_0\subset F\setminus G_{M_0},
 \qquad b:=\mu_{0,y}(F_0)>0.
\]
For every \(k\ge0\) and every \(Q\in\mathcal D_k\) meeting
\(F_0\), we have
\[
 \mu_{0,y}(Q)\le M_0(\diam Q)^q,
\]
since otherwise \(Q\subset G_{M_0}\), which is contrary to
\(Q\cap F_0\ne\varnothing\).

Define
\[
 \mu:=\frac{\mu_0\big|_{\pi_{T_yS^{n-1}}^{-1}(F_0)}}{b},
 \qquad
 \mu_y:=\bigl(\pi_{T_yS^{n-1}}\bigr)_{\#}\mu
        =\frac{\mu_{0,y}\big|_{F_0}}{b}.
\]
Then \(\mu\) is a probability measure supported on the compact set
\(Z\cap\pi_{T_yS^{n-1}}^{-1}(F_0)\). For every
\(k\ge0\) and \(Q\in\mathcal D_k\),
\[
 \mu_y(Q)
 =\frac{\mu_{0,y}(Q\cap F_0)}{b}
 \le\frac{M_0}{b}(\diam Q)^q,
\]
where the left-hand side is zero if \(Q\cap F_0=\varnothing\).
For any \(0<r\le1\), choose \(k\ge0\) with
\(r\le2^{-k}<2r\). A ball of radius \(r\) in the projection space
meets at most \(C_n\) cubes of \(\mathcal D_k\), whence
\[
 \mu_y(B(v,r))\le C_n2^q\frac{M_0}{b}r^q.
\]
The restriction also preserves the ambient Frostman bound:
\[
 \mu(B(z,r))\le\frac{C_0}{b}r^{\tau}.
\]
Finally, recalling from \eqref{qlower} that $\dim F -q < \epsilon$ and noting that $\dim F_0 \le \dim F$, we conclude the proof.
\end{proof}

Before the proof of Lemma \ref{disw}, the following remark is useful.
\begin{remark} \label{content}
   If $E$ is an analytic set and $\mu$ is an $s$-Frostman measure supported on $E$, then for any measurable $F \subset E$,
   \begin{equation}\label{content2}
     \calH^s_\infty(F) \gtrsim \mu(F).
   \end{equation}
Indeed, for any cover $\{Q_i\}_{i \in I}$ of $F$ by dyadic cubes, we have, since $\mu$ is $s$-Frostman, $\sum_{i \in I}(\diam Q_i)^s \gtrsim \sum_{i \in I}\mu(Q_i) \ge \mu(F)$. Taking the infimum among all covers of $F$, we obtain \eqref{content2}.
\end{remark}

We are ready to prove Lemma \ref{disw}.

\begin{proof}[Proof of Lemma \ref{disw}]
The dependence of $\delta_0$ on $\epsilon$ in this lemma is
\begin{equation}\label{d0}
  [\log(1/\delta_0)]^{-3} > \delta_0^\epsilon.
\end{equation}
Recalling $t=\dim Z> n-2$, choose \(n-2<{t'}<t\).
By Lemma \ref{frost}, there is a probability measure \(\mu\) supported
on $Z$ such that
\(\mu(B(z,r))\lesssim r^{t'}\) for \(0<r\le1\) and
\(\mu_y:=(\pi_{T_yS^{n-1}})_\#\mu\) is an $s$-Frostman measure with $\dim F - s < \epsilon$ where $F:=$ spt$\, (\mu_y)$.

Fix a large integer \(k_0\) and $s' \in (\dim F,\dim F+\epsilon)$. We have \(\calH^{s'}(F)=0\).
Choose a countable cover of $F$ by dyadic cubes \(\calQ\) with
\[
 F\subset\bigcup_{Q\in\calQ}Q,\qquad
 \diam Q\le2^{-k_0},\qquad
 \sum_{Q\in\calQ}(\diam Q)^{s'}\le1.
\]
Write
\[
 \calQ_k:=\{Q\in\calQ:\diam Q=2^{-k}\},\qquad
 U_k:=\bigcup_{Q\in\calQ_k}Q.
\]
By pigeonholing, there exists an integer \(k\ge k_0\) such that,
\[
 m:=\mu_y(U_k)\gtrsim k^{-2},\qquad
 N_k:=|\calQ_k|\le\delta^{-s'}
\]
where \(\delta:=2^{-k}\).
The cubes of \(\calQ_k\) are disjoint. Discard those with
\(\mu_y(Q)<m/(2N_k)\), losing at most half the mass, and pigeonhole
the remaining cubes according to their dyadic masses. There are only
\(O(1+\log(2N_k/m))=O(k)\) relevant classes. Thus a subfamily
\(\calQ_k'\subset\calQ_k\), with union \(U_k'\), satisfies, for some
\(\theta>0\),
\begin{equation}\label{equalmass}
   \frac\theta2<\mu_y(Q)\le\theta\quad(Q\in\calQ_k'),\qquad
 N:=|\calQ_k'|\le\delta^{-s'},\qquad
 \mu_y(U_k')\approx\theta N\gtrsim k^{-3}.
\end{equation}
For this union, since $\mu_y$ is $s$-Frostman, \eqref{content2} implies that
\[
 \calH^s_\infty(F\cap U_k')
   \gtrsim\mu_y(F\cap U_k')
   =\mu_y(U_k')\gtrsim k^{-3}= [\log(1/\delta)]^{-3} > \delta^\epsilon.
\]
Applying \cite[Proposition 2.4]{orponen2023hausdorff}, it gives a \(\delta\)-separated $(\delta,s,\delta^{-\epsilon})$-set
\(W\subset F\cap U_k'\) with
$
 \delta^{-s+\epsilon}\lesssim |W|\lesssim\delta^{-s}.
$

Let
\[
 \calQ(W):=\{Q\in\calQ_k':Q\cap W\ne\varnothing\},\qquad
 V:=\bigcup_{Q\in\calQ(W)}Q.
\]
Each \(\delta\)-cube contains only \(O_n(1)\) points of the separated
set \(W\); hence \(|\calQ(W)|\gtrsim_n|W|\). Also,
\(V\subset W^{\sqrt{n-1}\,\delta}\). Recalling \eqref{equalmass} and noting that $s'-s <2\epsilon$, we deduce that
\[
 \begin{aligned}
 \mu_y(V)
 &=\sum_{Q\in\calQ(W)}\mu_y(Q)
 \gtrsim\theta|W|\\
 &\gtrsim\frac{|W|}{N}\,\mu_y(U_k')
 \gtrsim k^{-3}\delta^{s'-s+\epsilon}
 =\delta^{5\epsilon}.
 \end{aligned}
\]
Define
\[
 A:=Z\cap\pi_{T_yS^{n-1}}^{-1}(V).
\]
Then \(\mu(A)=\mu_y(V)\gtrsim \delta^{5\epsilon}\). Since
\(\mu\) is \({t'}\)-Frostman, \eqref{content2} implies
\[
 \calH^{t'}_\infty(A)\gtrsim \delta^{5\epsilon}.
\]
Applying \cite[Proposition 2.4]{orponen2023hausdorff}, we obtain a \(\delta\)-separated $(\delta, t', \delta^{-5\epsilon})$-set \(Y\subset A\) with
$
 \delta^{-t'+5\epsilon}\lesssim|Y|\lesssim\delta^{-t'}$.

Finally, \(Y\subset A\) implies
\[
 \pi_{T_yS^{n-1}}(Y)\subset V\subset W^{\sqrt{n-1}\,\delta}
                      \subset W^{n\delta},
\]
which proves \eqref{eq:58}. For each \(z\in Y\), choose \(w\in W\)
with \(|\pi_{T_yS^{n-1}}(z)-w|<n\delta\). Since \(z\in B(o,1/2)\),
\(z\cdot y\in[-1/2,1/2]\), and therefore
\[
 w+(z\cdot y)y\in L_{w,y},\qquad
 |z-(w+(z\cdot y)y)|=|\pi_{T_yS^{n-1}}(z)-w|<n\delta.
\]
This proves \eqref{eq:59}. The proof is complete.
\end{proof}

\appendix

\section{Proof of Lemma \ref{int2}} \label{a1}
This appendix is intended to show Lemma \ref{int2}.
We first prove several auxiliary lemmas. All the conventions at the beginning of Section \ref{s4} are assumed.

\begin{remark} \label{bar}
  \rm By considering $\Sigma \subset \R^n$ is a codimension $2$ submanifold, we can also compute the principal curvatures $\bar \kappa_i$, $i=1, \cdots,n-2$
  of $\Sigma$ with respect to a unit normal $\nu$ in $\R^n$. We remark that $\bar \kappa_i = \kappa_i$ for all $i=1, \cdots,n-2$.
  This is because principal curvatures are eigenvalues of the second fundamental form.
  Thus $\{\kappa_i\}_{i=1, \cdots,n-2}$ are the eigenvalues of the following bilinear form
  \begin{equation}\label{nabla}
    T\Sigma\times T\Sigma \ni (X,Y) \mapsto \nabla_X Y \cdot \nu
  \end{equation}
  where $\nabla$ is the Levi-Civita connection on the standard sphere $S^{n-1}$
  and $\{\bar \kappa_i\}_{i=1, \cdots,n-2}$ are the eigenvalues of the following bilinear form
  \begin{equation}\label{nabla2}
    T\Sigma\times T\Sigma \ni (X,Y) \mapsto \bar\nabla_X Y\cdot \nu
  \end{equation}
  where $\bar\nabla$ is the Levi-Civita connection in $\R^n$.

  Since at each $y\in \Sigma$, we have the relation
  $$ \bar\nabla_X Y  =  \nabla_X Y -\langle X, Y \rangle y, $$
  by noting $y \perp \nu$, we know the forms \eqref{nabla} and \eqref{nabla2} coincide and $\bar \kappa_i = \kappa_i$ for all $i=1, \cdots,n-2$.
\end{remark}

\begin{lemma} \label{similar}
   For any point $y \in \calC \setminus \{o\}$, the principal curvatures of $\calC$ at $y$ are
   $$ 0, \frac{1}{|y|}\kappa_1(\frac{y}{|y|}),\cdots, \frac{1}{|y|}\kappa_{n-2}(\frac{y}{|y|})$$
   and the principal direction of the value $0$ is the direction of the generatrix passing through $y$, i.e. $\frac{y}{|y|}$.
\end{lemma}

\begin{proof}
   Since $\calC_w=\calC_o + w$ for all $w \in \R^n$, and translations in $\R^n$ are isometries which keep the principal curvatures, it suffices to show the case $w=o$ and we write $\calC=\calC_o$.

   Note that for any $y \in \calC$, the generatrix $L_{o,\frac{y}{|y|}}$ which is a line segment is contained in $\calC$, thus the principal curvature with respect to this direction is $0$.

     To see the other principal curvatures, note that we can write $\calC$ as
     \begin{equation}\label{sigmar}
       \calC \setminus \{o\}= \bigcup_{r\in(0,1]} {\bf r}^{-1}(\Sigma) \cap S^{n-1}_r =: \bigcup_{r\in(0,1]} \Sigma_r
     \end{equation}
  where ${\bf r}^{-1}$ is the preimage set of the radial projection and $S^{n-1}_r$ is the sphere centred at $o$ with radius $r$.
  Thus $\Sigma_r \subset S^{n-1}_r$ is the image of $\Sigma$ under the scaling map $x \mapsto rx$, $x \in \R^n$, which implies the principal curvatures at $y \in \Sigma_{|y|}$ is $\frac{1}{|y|}$ times the principal curvatures at $\frac{y}{|y|}\in\Sigma_{1}=\Sigma$.

  It suffices to show the lemma for $y \in \Sigma \subset \calC$.
  Recall that at $y  \in \Sigma$, we have the relation
  $$  \R^n= \mbox{span}\{y,\nu\} \oplus T_y\Sigma$$
  and $T_y\calC= \mbox{span}\{y\} \oplus T_y\Sigma$ where $\nu$ is the unit normal of $\Sigma \subset S^{n-1}$. Thus $\nu$ is also the unit normal of $\calC \subset \R^n$. Since we already know $\mbox{span}\{y\}$ is one of the  principal direction with principal curvature $0$, it suffices to find principal curvatures with
   the  principal directions in $T_y\Sigma \subset T_y\calC$, which are
the eigenvalues of the second fundamental form $\nu$ restricted to $T_y\Sigma$. Using remark \ref{bar}, we conclude the desired result.
\end{proof}

\begin{remark} \label{graph}
   We recall a fact that any $n$-manifold $M$ can be seen as a graph of a function $h$ from $T_xM =\R^n$ for any $x \in M$ locally such that $\nabla h(x)=0$. And moreover, the eigenvalues of the Hessian $\nabla^2 h(x)$ at $x$ coincide with eigenvalues of the second fundamental form of $M$ at $x$ which are precisely the principal curvatures of $M$ at $x$.
\end{remark}

\begin{lemma} \label{small}
  The function $h : U \to \R_+$ with graph$(h)=\calC$ satisfies that the Hessian of $h$ at any $x \in U\setminus\{o\}$ has one eigenvalue $0$ and all other eigenvalues $> \frac{\underline \kappa}{2|x|}$. In particular, $h$ is convex.
\end{lemma}

\begin{proof}
  By the convention at the beginning of this section, there is a generatrix $L_{o,p}$ of $\calC$ with direction $e_1$ where $p \in \Sigma$. In particular, $L_{o,p} \subset U$. We first show the lemma for points in this generatrix.

  Let $x \in L_{o,p}$. Then by our convention at the beginning of this section, we know $T_x\calC =$ span$\{e_1,\cdots,e_{n-1}\}$. Thus at $x$, $\calC$ is the graph of the function $h$ from its tangent space. By Remark \ref{graph} and Lemma \ref{similar}, we know at $x \in \calC$, the Hessian of $h$ has one eigenvalue $0$ and all other eigenvalues $\ge \frac{\underline \kappa}{|x|}$. Thus the Lemma holds for $x \in L_{o,p}$.

  We then prove the lemma for other points. To this end, consider the following decomposition of $U$. Recalling from \eqref{sigmar} that
  $$ \calC = \bigcup_{r \in [0,1]} \Sigma_r.$$
  For each $r \in [0,1]$, let $U_r$ be the orthogonal projection of $\Sigma_r$ to span$\{e_1,\cdots,e_{n-1}\}$. Then, we know
  $$  U = \bigcup_{r \in [0,1]} U_r.$$
  Let $y \in U_1$.
  Since $\Sigma$ is $C^2$, we know the eigenvalues are continuously depending on $z \in U\setminus\{o\}$, thus if $\Sigma$ is chosen sufficiently small, we know $U_1$ is also small such that all points $y \in U_1$ are sufficiently close to $U_1 \cap L_{o,p}$.
  As a result, the Hessian of $h$ has one eigenvalue $0$ and all other eigenvalues $\ge \frac{\underline \kappa}{2}$ for all $y \in U_1$.

  For other points $z \in U_r$, since $\frac{z}{|z|} \in U_1$, by Lemma \ref{similar}, we deduce the Hessian of $h$ has one eigenvalue $0$ and all other eigenvalues $\ge \frac{\underline \kappa}{2|z|}$ for all $z \in U_r$. The proof is complete.
\end{proof}

\begin{remark} \label{Uconv}
We remark that since $h:U \to \R_+$ is convex. One can show that the epigraph $E=\{(x,y) \in U\times \R: y \ge h(x)\} \subset \R^n$ is a convex set. Since the domain $U= \pi_{e_n^\perp}(E)$, we know $U \subset \R^{n-1}$ is also a convex set.
\end{remark}

We are ready to show Lemma \ref{int2}. We recall it here.

\medskip
\noindent
\textbf{Lemma 3.1. }
For every line $L\subset\mathbb R^n$, the set
\[
L\cap \calC^\delta
\]
has at most two connected components.

\begin{proof}
Write
\[
\calC=\operatorname{graph}(h)
=
\{(u,h(u)):u\in U\}.
\]

We now fix a line $L\subset\mathbb R^n$. If $L$ is parallel
to $e_n$, then, for some fixed $u_0\in\mathbb R^{n-1}$,
\[
L=\{(u_0,t):t\in\mathbb R\}.
\]
If $u_0\notin U$, then $L\cap C^\delta=\varnothing$. If $u_0\in U$,
then
\[
L\cap C^\delta
=
\{(u_0,t):|t-h(u_0)|<\delta\},
\]
which is an interval. Thus the conclusion holds in this case.

Suppose that $L$ is not parallel to the vertical axis. We may
parametrise $L$ as
\[
L
=
\{(u_0+sv,t_0+s\tau):s\in\mathbb R\},
\]
where $v\in\mathbb R^{n-1}\setminus\{0\}$ and $\tau\in\mathbb R$.
Let
\[
J:=\{s\in\mathbb R:u_0+sv\in U\}.
\]
Since by Remark \ref{Uconv}, $U$ is convex, $J$ is an interval. Define
\[
\varphi(s)
:=
t_0+s\tau-h(u_0+sv),
\qquad s\in J.
\]
Then
\[
L\cap C^\delta
=
\{(u_0+sv,t_0+s\tau):s\in J,\ |\varphi(s)|<\delta\}.
\]
Since $h$ is convex, $\varphi$ is concave.

For a concave function on an interval, both sets
\[
I_1:=\{s\in J:\varphi(s)> -\delta\}
\]
and
\[
I_2:=\{s\in J:\varphi(s) \ge \delta\}
\]
are intervals, possibly empty. We have
\[
\{s\in J:|\varphi(s)|<\delta\}
=
I_1\setminus I_2.
\]
The difference of one interval by another interval has at most two
connected components. Therefore
\[
\{s\in J:|\varphi(s)|<\delta\}
\]
has at most two connected components. The proof is complete.
\end{proof}

\bibliographystyle{plain}
\bibliography{references.bib}

\end{document}